\documentclass[11pt,twoside]{article}

\usepackage{amsmath, amsthm,amsfonts,amssymb,mathtools,mathrsfs,bbold}
\usepackage[utf8]{inputenc}
\usepackage[T1]{fontenc}
\usepackage{mlmodern, tikz-cd, quiver}

\usepackage{enumitem}
\usepackage[hidelinks]{hyperref}
\usepackage{tikz, fancyhdr, xparse, xcolor, tocloft}
\usepackage[british]{babel}
\usepackage[a4paper, top=4.5cm, bottom=4.5cm, left=4cm, right=4cm, asymmetric]{geometry}

\usepackage{ahtitle, titlesec}
\usepackage{etoolbox}
\makeatletter
\patchcmd{\ttlh@hang}{\parindent\z@}{\parindent\z@\leavevmode}{}{}
\patchcmd{\ttlh@hang}{\noindent}{}{}{}
\makeatother

\newcommand\eqdef{\coloneqq}
\newcommand\nbd{\nobreakdash-\hspace{0pt}}
\newcommand\idd[1]{\mathrm{id}_{#1}}
\newcommand\incl{\hookrightarrow}
\newcommand\set[1]{\left\{ {#1} \right\}}
\newcommand{\overbar}[1]{{\mkern 1.5mu\overline{\mkern-1.5mu#1\mkern-1.5mu}\mkern 1.5mu}}
\newcommand\filtr[2]{#2^{(#1)}}
\newcommand\NN{\mathbb{N}}
\newcommand\ZZ{\mathbb{Z}}
\newcommand\size[1]{\left|{#1}\right|}

\newcommand\cat[1]{\underline{\mathrm{#1}}}
\newcommand\fun[1]{\mathsf{#1}}
\newcommand{\Set}{\cat{Set}}
\newcommand{\Cat}{\cat{Cat}}
\newcommand\sSet{\cat{sSet}}
\newcommand{\omegaCat}{\omega\cat{Cat}}
\newcommand\cp[1]{\,{\scriptstyle\#}_{#1}\,}   
\newcommand\oper[1]{\mathcal{#1}} 	
\DeclareMathOperator{\Nd}{Nd}
\newcommand\bd[2]{\partial_{#1}^{#2}}
\newcommand\skel[1]{\sigma_{\leq {#1}}}
\newcommand\trunc[1]{\tau_{\leq {#1}}}
\newcommand\afib{\overset{\sim}{\to}}
\newcommand\sus[1]{\fun{S}{#1}}

\newcommand\globe[1]{O^{#1}}               

\newcommand{\FF}{\mathbb{F}_2}
\newcommand{\NCPop}{\oper{NCP}}         
\newcommand{\NCP}{\mathrm{NCP}}
\newcommand{\PNCPop}{\oper{NCP}_{\FF}}
\newcommand{\PNCP}{\mathrm{NCP}_{\FF}}  
\newcommand{\Bop}{\oper{B}}             
\newcommand{\pcirc}{\mathbin{\circ}}    
\newcommand{\unit}{\mathbb{1}}
\newcommand{\Nshift}[2]{\fun{s}_{#1}^{#2}}	
\newcommand{\levNCP}{\ell_{\NCP}}	
\newcommand{\levB}{\ell_{\mathrm{B}}}		
\newcommand{\lev}{\ell}			
\newcommand{\coarse}{\psi}		

\newcommand{\der}{\mathrm{d}}             
\newcommand{\derbar}{\overbar{\der}}	  
\newcommand{\aug}{\mathrm{e}}             
\newcommand{\chaug}{\cat{Ch}^+}           
\newcommand{\linea}[1]{\fun{\lambda}{#1}} 
\newcommand{\nufun}[1]{\fun{\nu}{#1}}     
\DeclareMathOperator*{\cosk}{cosk}        
\newcommand{\gener}[1]{\mathscr{#1}}      
\newcommand{\NSimp}{\fun{N}}
\newcommand\spanset[1]{\langle {#1} \rangle}
\newcommand\orient{\mathscr{O}}

\renewcommand{\a}{\alpha}
\renewcommand{\b}{\beta}
\newcommand{\lightregion}{gray!20}
\newcommand{\darkregion}{black!90}

\newcommand\omegat{\texorpdfstring{$\omega$}{omega}}

\newcommand\ncht{2}                       
\newcommand\pncp[1]{\,\raisebox{-1ex}{%
  \tikz[x=1.8ex,y=1.2ex,line width=0.7pt,line cap=round]{#1}}\,}
  \NewDocumentCommand\ncstick{O{0} m}{\draw[-] (#2,#1)--(#2,\ncht);}
  \NewDocumentCommand\nccup{O{0} m m}{\draw[-] (#2,\ncht)--(#2,#1)--(#3,#1)--(#3,\ncht);}
  \newcommand\ncpol[1]{\draw[-] (#1,\ncht)--(#1-0.5,\ncht);}
  \newcommand\ncpolfin[1]{\draw[-] (#1+1,\ncht)--(#1+0.5,\ncht);}

\newtheoremstyle{ittheorem}
  {\topsep}   
  {\topsep}   
  {\itshape}  
  {0pt}       
  {\sffamily \itshape \bfseries} 
  { ---}         
  {5pt plus 1pt minus 1pt} 
  {}          

\newtheoremstyle{itdfn}
  {\topsep}   
  {\topsep}   
  {}  
  {0pt}       
  {\sffamily \itshape \bfseries} 
  {}         
  {5pt plus 1pt minus 1pt} 
  {\thmnumber{#2}{\thmnote{\normalfont\ \ %
  {\sffamily(#3)}.}}}          

\newtheoremstyle{itrmk}
  {0.5\topsep}   
  {0.5\topsep}   
  {\normalfont}  
  {0pt}       
  {\sffamily \itshape} 
  { ---}         
  {5pt plus 1pt minus 1pt} 
  {}          
  
\makeatletter
  \renewcommand\@upn{\textit}
\makeatother

\theoremstyle{ittheorem}
\newtheorem{thm}{Theorem}[section]
\newtheorem{prop}[thm]{Proposition}
\newtheorem{lem}[thm]{Lemma}
\newtheorem{conj}[thm]{Conjecture}
\theoremstyle{itdfn}
\newtheorem{dfn}[thm]{}
\theoremstyle{itrmk}
\newtheorem{rmk}[thm]{Remark}
\newtheorem{comm}[thm]{Comment}
\newtheorem{exm}[thm]{Example}
\newtheorem{evid}[thm]{Evidence}

\setlist{leftmargin=20pt,itemsep=0pt,parsep=0pt,topsep=1ex}

\titleformat{\section}
 {\large\scshape}{\thesection.}{1em}{}

\titleformat{\subsection}
 {\normalsize\itshape}{\thesubsection.}{1em}{}
\titlespacing*{\subsection}
{0pt}{1.5ex plus 1ex minus .2ex}{1.5ex plus .2ex}

\renewcommand{\cftsubsecpagefont}{\mdseries}

\makeatletter \renewcommand{\cftsubsecfillnum}[1]{%
  {\cftsubsecleader}\nobreak
  \makebox[\@pnumwidth][\cftpnumalign]{\cftsubsecpagefont \oldstylenums{#1}}\cftsubsecafterpnum\par
} \makeatother
\makeatletter \renewcommand{\cftsecfillnum}[1]{%
  {\cftsecleader}\nobreak
  \makebox[\@pnumwidth][\cftpnumalign]{}\cftsecafterpnum\par
} \makeatother

\AtBeginDocument{\addtocontents{toc}{\protect\thispagestyle{fancy}}}

\newcommand\runtitle{polarised ncps and the self-dual \omegat-equivalence}
\newcommand\runauthor{hadzihasanovic}

\title{Polarised noncrossing partitions 
\\and the coherent self-dual \omegat-equivalence}

\author{Amar Hadzihasanovic}

\institution{Tallinn University of Technology}

\begin{document}

{$\quad$}

\vspace{20pt}

\maketitle

\noindent\makebox[\textwidth][r]{%
	\begin{minipage}[t]{.7\textwidth}
\small \emph{Abstract.}
We construct an acyclic augmented chain complex of abelian groups whose entry in degree $n > 0$ is free on the set of noncrossing partitions of degree $n-1$ equipped with a $\set{0, 1}$-labelling of their gaps.
The definition of the differential in this complex is related, via a restricted Leibniz rule, to the gap-insertion operad of Ebrahimi-Fard, Foissy, Kock, and Patras.
We conjecture that this augmented chain complex is the linearisation of a polygraph presenting a \emph{self-dual} model of the coherent walking $\omega$\nbd equivalence constructed by the author, Loubaton, Ozornova, and Rovelli, and provide evidence for this conjecture.
\end{minipage}}

\vspace{20pt}

\makeaftertitle

\normalsize \thispagestyle{empty}

\section*{Introduction}

\noindent
Part of any sufficiently structured model of higher categories is a specification of what it means for two objects in a higher category to be \emph{equivalent} \cite{ozornova2024equivalence}.
In models of homotopy-theoretic provenance, the natural answer is often extrinsic: two objects are equivalent if they are connected by a path in some underlying space.
In algebraic models, a most natural answer---definable in the algebra of cells, units, and composition---is provided by the notion of \emph{$\omega$\nbd equivalence}, also known as pseudoinvertibility or weak invertibility \cite{cheng2007omega}.
This postulates that two objects $s, t$ are equivalent if they admit a pair of morphisms $f\colon s \to t$, $g\colon t \to s$, such that the composite $gf$ is itself equivalent to $\idd{s}$ and the composite $fg$ is equivalent to $\idd{t}$ in the higher categories $\mathrm{hom}(s, s)$, $\mathrm{hom}(t, t)$, respectively.
More in general, there are ``truncated'' notions of \emph{$(n-1)$\nbd equivalence} between $n$\nbd categories for each $n \in \NN$. 
This notion of equivalence has recently received renewed attention both in strict \cite{hadzihasanovic2025model} and in weak \cite{fujii2024weakly, benjamin2024invertible, chanavat2025equivalences} models of higher categories.

For computational purposes, it is useful to conceive of weak invertibility as \emph{structure} on a morphism; that is, consider that an $\omega$\nbd equivalence $f\colon s \to t$ comes equipped with a weak inverse $g\colon t \to s$ and higher cells witnessing the inversion.
In particular, it is desirable that such data can be ``packaged'' in a higher category---the free higher category containing an $\omega$\nbd equivalence, also known as \emph{the walking $\omega$\nbd equivalence}
---since this provides an explicit construction for the \emph{localisation} of a higher category at a morphism: to localise, attach a copy of the walking $\omega$\nbd equivalence.

As discussed in \cite{lack2004quillen} in the case of equivalences of 2-categories, in order for such an invertibility structure to be essentially unique, it is essential that the walking equivalence be \emph{contractible} in the appropriate homotopy theory.
However, already in the setting of \emph{strict} $\omega$\nbd categories---which can be construed as a ``toy model'' of higher category theory, with the advantage that more computational tools are currently available compared to fully weak and homotopical models---the ``naive'' version of the walking $\omega$\nbd equivalence fails to be contractible in the appropriate homotopical sense, as provided by the folk model structure of Lafont, M\'etayer, and Worytkiewicz \cite{lafont2010folk}.
In fact, already the ``naive walking equivalence'' of 2\nbd categories is not contractible \cite[Section 1.2.4]{ozornova2024equivalence}.

In \cite[Section 1.5.4]{ozornova2024equivalence}, Ozornova and Rovelli proposed a candidate model for the \emph{coherent} (that is, contractible) walking $\omega$\nbd equivalence of strict $\omega$\nbd categories, also inducing models of coherent walking $(n-1)$\nbd equivalences of strict $n$\nbd categories. 
This was eventually proven to be contractible by Ozornova and Rovelli jointly with the author and Loubaton in \cite{hadzihasanovic2025model}.

While this solved the question from a homotopical standpoint, from a \emph{computational} standpoint this model has a downside.
The notion of $\omega$\nbd equivalence is manifestly symmetric: if $s$ is equivalent to $t$, then $t$ is equivalent to $s$.
More graphically, one would imagine an equivalence between $s$ and $t$ to be something like an ``undirected'' edge, a path that can be traversed both ways; so one would expect that the data of an equivalence should be \emph{self-dual} under the reversal of $s$ and $t$.
However, the Ozornova--Rovelli model is asymmetric: it dodges the non-coherence problem by separately assigning both a \emph{left} inverse $g^L$ and a \emph{right} inverse $g^R$ to $f$.

The naive walking $\omega$\nbd equivalence is self-dual, but not coherent.
A coherent self-dual model of the walking 1\nbd equivalence exists: it is the \emph{walking adjoint equivalence}.
There is also a conjecturally coherent self-dual model of walking 2\nbd equivalence, due to Gurski \cite{gurski2012biequivalences}.
To our knowledge, no self-dual coherent models for higher dimensions appear in the literature.

This article aims to be a first step towards the construction of an explicit self-dual model of the coherent $\omega$\nbd equivalence.
The initial difficulty in building such a model is simply to come up with an enumeration of the necessary coherence cells in each dimension.
The article presents an \emph{ansatz} for the combinatorics underlying such an enumeration, connecting them to objects that we call \emph{polarised noncrossing partitions}: these are noncrossing partitions \cite{kreweras1972partitions} together with a labelling of their \emph{gaps} in a 2\nbd element set, say, $\set{0, 1}$; that is, pairs $(p, b_1 \cdots b_{n+1})$ of a noncrossing partition of degree $n$ and a string of $n+1$ bits.

We do not know where the connection comes from, as it was largely reverse-engineered:
we first enumerated plausible ``minimal'' coherences up to dimension 5, and noticed that the sequence of their numbers matched the sequence $2^n C_{n-1}$ for $n > 0$, where $C_k$ is the $k$\nbd th Catalan number.
Of course, Catalan numbers are associated to many combinatorial structures, so why noncrossing partitions?
The key observation is that noncrossing partitions carry an operadic structure---the \emph{gap-insertion operad} of Ebrahimi-Fard, Foissy, Kock, and Patras \cite{ebrahimi2020operads}---and that it can be extended to polarised noncrossing partitions.
Using this structure, we can express each polarised noncrossing partition in a canonical form as a composite of certain generating elements, what we call its \emph{rightmost decomposition}.
Defining a differential on the generators, then extending along rightmost decompositions with a Leibniz-like rule, we can build an \emph{augmented chain complex} $E_\bullet$ whose entries in degree $n > 0$ are free on the set of polarised noncrossing partitions of degree $n-1$.

We know from Steiner theory \cite{steiner2004omega, ara2020joint} that a polygraph $X$---a strict $\omega$\nbd category that is dimension-wise freely generated by a set of cells \cite{ara2025polygraphs}---induces an augmented chain complex $\linea{X}_\bullet$ which in degree $n$ is free on the set of generating $n$\nbd cells of $X$; the differential applied to a generating $n$\nbd cell $x$ yields the sum of generating $(n-1)$\nbd cells appearing in the \emph{target} of $x$, minus the sum of generating $(n-1)$\nbd cells appearing in the \emph{source} of $x$, counted with their multiplicity.

It so happens that, in low dimensions, $E_\bullet$ reproduces exactly the linearisation of the coherence cells that we constructed by hand.
Moreover, $E_\bullet$ satisfies several conditions that are necessary or expected for such a linearisation:
\begin{enumerate} 
	\item it is self-dual, that is, it has a non-trivial involutive automorphism exchanging the two generators in degree 0 (Proposition \ref{prop:chain_selfdual}); 
	\item it is acyclic, that is, contractible in the sense appropriate for chain complexes (Theorem \ref{thm:acyclicity});
	\item it admits a filtration of acyclic subcomplexes $\filtr{k}{E}_\bullet$, $k \in \NN$, that for $k = 0, 1$ we can lift to explicit polygraphs $k E$ such that $\linea{(k E)}_\bullet$ is isomorphic to $\filtr{k}{E}_\bullet$, and such that $1 E$ is related to the ``walking isomorphism'' category.
\end{enumerate}
We thus conjecture that $E_\bullet$ is the linearisation of a model $\omega E$ of the coherent self-dual $\omega$\nbd equivalence, and, more generally, $\filtr{n}{E}_\bullet$ is the linearisation of a strict $\omega$\nbd category $n E$ whose ``intelligent'' $n$\nbd truncation $\trunc{n}{(n E)}$ is a model of the coherent self-dual $(n-1)$\nbd equivalence (Conjecture \ref{conj:main_conjecture}).
We provide evidence for this conjecture in the form of explicit constructions of certain infinite families of cells of $\omega E$, as well as a sample of low-dimensional instances of those that are not covered by these cases, comparing them to the known low-dimensional models of coherent $n$\nbd equivalences.

While we stop short of solving the problem, we find that the construction of $E_\bullet$ is in itself interesting and non-trivial, and that the unexpected point of contact revealed between operads arising in combinatorics and higher categories via chain complexes is noteworthy, and might invite further exploration of their connections.
The form of our definition of a chain complex from an operad---with a differential that reduces the \emph{arity} of operations, instead of the size of trees, and satisfies a Leibniz rule with respect to a particular decomposition, instead of summing over all decompositions---is unusual, and we could not find analogous prior examples.
Similarly, for the methodology of attempting to lift an augmented chain complex which is \emph{not} a Steiner complex to a polygraph, seeing it as a ``linearised'' shadow of a higher-categorical structure, we could not find antecedents in the literature.
Thus, we believe that the mathematical and methodological questions that this work raises are at least as valuable as the partial answers that it offers. 

\subsection*{Acknowledgements}

The initial idea for this work was developed during my visit to Viktoriya Ozornova and F\'elix Loubaton at MPIM Bonn in February 2024, supported by the Estonian Research Council grant PSG764.
It was then set aside for a long time as faster progress was made in the direction of \cite{hadzihasanovic2025model}, and resurrected after discussions with Ozornova and Thibaut Benjamin at the \emph{Formalizing Higher Categories} event at Institut Mittag-Leffler in June 2026, which I attended with support from the ARIA Safeguarded AI programme.
I am grateful to Cl\'emence Chanavat for feedback on an earlier draft.

\section{The operad of polarised noncrossing partitions} \label{sec:polarised}

\noindent
We fix our notation and recall some facts about noncrossing partitions and the gap-insertion operad from \cite{ebrahimi2020operads}.
Throughout, given $n \in \NN$, a \emph{partition of degree $n$} will be a partition $\set{\pi_1, \ldots, \pi_k}$ of the set $\set{1, \ldots, n}$.

\begin{dfn}[Noncrossing partition]
Let $n \in \NN$.
A partition $\set{\pi_1, \ldots, \pi_k}$ of degree $n$ is \emph{noncrossing} if, for all $i, j \in \set{1, \ldots, k}$, all $a, a' \in \pi_i$, and all $b, b' \in \pi_j$, if $a < b < a' < b'$, then $i = j$.
\end{dfn}

\noindent
Let $\NCP(n)$ denote the set of noncrossing partitions of degree $n$.
Then $\size{\NCP(n)} = C_n$, the $n$-th Catalan number.

\begin{rmk}
All partitions of degree $n \leq 3$ are noncrossing.
The first non-example is the partition $\set{\set{1, 3},\set{2, 4}}$ of degree 4.
\end{rmk}

\noindent
We will adopt the common depiction of noncrossing partitions in terms of \emph{block diagrams}, which is best described by way of example:
\begin{equation*}
	\set{ \set{1, 2, 5}, \set{3, 4}, \set{6, 8}, \set{7} }
	\quad \mapsto \quad
	\pncp{ \nccup{1}{5}\ncstick{2}\nccup[0.9]{3}{4}\nccup{6}{8}\ncstick[0.9]{7}
	}\;.
\end{equation*}
The noncrossing property corresponds to the fact that the blocks do not cross each other; notice that the non-example $\set{\set{1, 3}, \set{2, 4}}$ is depicted as $\pncp{\nccup{1}{3}\nccup[0.9]{2}{4}}$.

\begin{dfn}[Operad] 
A (non-symmetric, single-coloured, set) \emph{operad} is a graded set $\coprod_{n \in \NN} \oper{P}(n)$ equipped with a unit operation $\unit \in \oper{P}(1)$, as well as partial composition operations
\begin{equation*}
	- \pcirc_i - \colon \oper{P}(n) \times \oper{P}(m) \to \oper{P}(n + m - 1), \quad i \in \set{1, \ldots, n}
\end{equation*}
for each $n, m \in \NN$, satisfying the axioms
\begin{align*}
	& x = \unit \pcirc_1 x = x \pcirc_i \unit && \text{(unitality)}, \\
	& x \pcirc_i (y \pcirc_j z) = (x \pcirc_i y) \pcirc_{i+j-1} z && \text{(sequential associativity)}, \\
	& (x \pcirc_{i+k} z) \pcirc_i y = (x \pcirc_i y) \pcirc_{i+m-1+k} z && \text{(parallel associativity)}
\end{align*}
for all $x \in \oper{P}(n)$, $y \in \oper{P}(m)$, $z \in \oper{P}(\ell)$, $i \in \set{1, \ldots, n}$, $j \in \set{1, \ldots, m}$, and $k \in \set{1, \ldots, n-i}$.
\end{dfn}

\noindent
Given an operad $\oper{P}$, the elements of $\oper{P}(n)$ are usually called \emph{$n$\nbd ary operations}.
In the gap-insertion operad, a noncrossing partition $p$ of degree $n$ determines an $(n+1)$\nbd ary operation as follows: as a block diagram, $p$ has $(n-1)$ \emph{inner gaps} between vertical lines, as well as two \emph{outer gaps} before the leftmost and after the rightmost line.
Indexing these gaps by $i \in \set{1, \ldots, n+1}$ from the leftmost to the rightmost, we let $p \pcirc_i q$ be the noncrossing partition obtained from the insertion of $q$'s block diagram into the $i$\nbd th gap of $p$'s block diagram:
\begin{align*}
	\pncp{ \nccup{1}{3}\ncstick{2} } \pcirc_1 \pncp{ \nccup{1}{2} } 
		& = \pncp{ \nccup{1}{2} \nccup{3}{5} \ncstick{2} }, 
	&\pncp{ \nccup{1}{3}\ncstick{2} } \pcirc_2 \pncp{ \nccup{1}{2} } 
		& = \pncp{ \nccup{1}{5} \nccup[0.9]{2}{3} \ncstick{4} }, \\
	\pncp{ \nccup{1}{3}\ncstick{2} } \pcirc_3 \pncp{ \nccup{1}{2} } 
		& = \pncp{ \nccup{1}{5} \nccup[0.9]{3}{4} \ncstick{2} },
	&\pncp{ \nccup{1}{3}\ncstick{2} } \pcirc_4 \pncp{ \nccup{1}{2} } 
		& = \pncp{ \nccup{1}{3} \ncstick{2} \nccup{4}{5} }.
\end{align*}
More formally, for each $m, i \in \NN$, let $\Nshift{i}{m}\colon \NN \to \NN$ be defined by $n \mapsto n$ if $n < i$, and $n \mapsto n + m$ if $n \geq i$.
 
\begin{dfn}[Gap-insertion operad]
	The \emph{gap-insertion operad} $\NCPop$ has sets of operations $\NCPop(0) \eqdef \varnothing$ and $\NCPop(n) \eqdef \NCP(n-1)$ for $n > 0$, with the empty partition $\set{}$ of degree 0 as unit operation, and partial composition defined by
\begin{equation*}
	\set{\pi_1, \ldots, \pi_k} \pcirc_i \set{\rho_1, \ldots, \rho_\ell} \eqdef \set{\Nshift{i}{m}(\pi_1), \ldots, \Nshift{i}{m}(\pi_k), \Nshift{1}{i-1}(\rho_1), \ldots, \Nshift{1}{i-1}(\rho_\ell)}
\end{equation*}
for all noncrossing partitions $\set{\pi_1, \ldots, \pi_k}$ of degree $n$, $\set{\rho_1, \ldots, \rho_\ell}$ of degree $m$, and all $i \in \set{1, \ldots, n+1}$.
\end{dfn}

\noindent
For each $n > 0$, let $\coarse_n$ be the coarsest partition of degree $n$, $\set{ \set{1, \ldots, n} }$.
By \cite[Proposition 3.1.4]{ebrahimi2020operads}, $\set{\coarse_n \mid n > 0}$ is a minimal set of generators for $\NCPop$.
We refine this to a normal form result, which we will use later on.

\begin{dfn}[Last block]
	Let $p = \set{\pi_1, \ldots, \pi_k}$ be a noncrossing partition of degree $n > 0$, and, reindexing if necessary, let the blocks of $p$ be ordered so that
	\[
		\min \pi_i < \min \pi_j \text{ if and only if } i < j.
	\]	
	Then the \emph{last block} of $p$ is $\pi_k$.
\end{dfn}

\begin{dfn}[Level of a noncrossing partition]
	Let $p$ be a noncrossing partition of degree $n$.
	The \emph{level} of $p$ is the natural number
	\begin{equation*}
		\levNCP(p) \eqdef \begin{cases}
			0 
			& \text{if $n = 0$}, \\
			\min \pi_k
			& \text{if $n > 0$ and $\pi_k$ is the last block of $p$.}
		\end{cases}
	\end{equation*}
\end{dfn}

\begin{lem} \label{lem:final_block_factorisation}
	Let $p = \set{\pi_1, \ldots, \pi_k}$ be a noncrossing partition of degree $n > 0$ and let $i \eqdef \levNCP(p)$.
	Then there exists a unique pair of a noncrossing partition $p' = \set{\pi'_1, \ldots, \pi'_{k-1}}$ and $m > 0$ such that $p = p' \pcirc_i \coarse_m$.
	Moreover, $\levNCP(p') < \levNCP(p)$.
\end{lem}
\begin{proof}
Suppose $\pi_k$ is the last block of $p$, and let $m \eqdef \size{\pi_k}$.
We claim that $\pi_k$ is the entire interval $[i, i+m-1]$.
By definition of $\levNCP(p)$, we have $i \in \pi_k \subseteq [i, n]$.
If $n = i+m-1$, then we are done.
Else, let $j > i$ be minimal such that $j \not\in \pi_k$, that is, $j \in \pi_\ell$ for some $\ell \neq k$; we claim that $j = i+m$.
Indeed, suppose by way of contradiction that $j < i+m$.
Then there exists $j' \in \pi_\ell$ with $j' < i < j$ because $\pi_k$ is the last block of $p$, and there exists $i' \in \pi_k$ with $i' > j$ because $\pi_k \not\subseteq [i, j-1]$.
Then $j' < i < j < i'$ implies $\ell = k$ by the noncrossing property, a contradiction.

It follows that $\bigcup_{j < k} \pi_j = [1, i-1] \cup [i+m, n]$, and the function $\Nshift{i}{m}$ restricted to $\set{1, \ldots, n-m}$ determines a bijection with this set, whose inverse we denote by $\fun{t}$.
Then $p' \eqdef \set{ \fun{t}(\pi_1), \ldots, \fun{t}(\pi_{k-1}) }$ is a noncrossing partition of degree $n-m$ and strictly lower level, and by construction $p' \pcirc_i \coarse_m = p$. 
Moreover, the minimum of the final block of $p'$ corresponds to the minimum of a non-final block of $p$, so $\levNCP(p') < \levNCP(p)$.
This proves existence.
For uniqueness, observe that, in any factorisation $p = p' \pcirc_i \coarse_m$, the condition $i = \levNCP(p)$ forces the single block of $\coarse_m$ to become the last block of $p$, and as in the existence proof, this fact completely determines the factorisation.
\end{proof}

\begin{prop} \label{prop:ncp_unique_factorisation}
	Let $p = \set{\pi_1, \ldots, \pi_k}$ be a noncrossing partition.
	Then there exists a unique factorisation of the form
	\begin{equation*}
		p = (\ldots(\set{} \pcirc_{i_1} \coarse_{m_1}) \pcirc_{i_2} \coarse_{m_2} \ldots) \pcirc_{i_k} \coarse_{m_k}
	\end{equation*}
	with $i_1 < \ldots < i_k$.
	Moreover, $\levNCP(p) = \max \set{0, i_1, \ldots, i_k}$.
\end{prop}
\begin{proof}
	Existence follows by well-founded recursion from Lemma \ref{lem:final_block_factorisation}, using either the number of blocks or $\levNCP(-)$ as a strictly decreasing index. 
	Uniqueness also follows after observing that the condition that the $i_j$ be strictly ascending forces each subfactor ending in $\pcirc_{i_j} \coarse_{m_j}$ to have level $i_j$.
\end{proof}

\noindent 
Let $\FF$ denote the field with two elements $\set{0, 1}$, that we call \emph{bits}, and let $\oplus$ denote its addition.
\begin{dfn}[Polarised noncrossing partition]
	Let $n \in \mathbb{N}$.
	A \emph{polarised noncrossing partition} is a pair $(p, b_1 \cdots b_{n+1})$ of a noncrossing partition $p$ of degree $n$ and a string of $n+1$ bits $b_i \in \set{0, 1}$.
\end{dfn}

\noindent
The idea is that, for each $i \in \set{1, \ldots, n+1}$, the $i$\nbd th gap of $p$ is endowed with the ``polarity'' $b_i$.
We extend the block diagram representation to polarised noncrossing partitions as follows: for all $i \in \set{1, \ldots, n}$, the vertical line to the right of the $i$\nbd th gap extends horizontally to its top left if and only if $b_i = 1$; if $b_{n+1} = 1$, we leave a floating horizontal dash in the rightmost gap, extending a ``virtual'' vertical line to its right.
For example,
\begin{align*}
	\left( \set{ \set{1, 2, 5}, \set{3, 4} }, 010110 \right) 
	\quad & \mapsto \quad
	\pncp{ \nccup{1}{5} \nccup[0.9]{3}{4} \ncstick{2} \ncpol{2} \ncpol{4} \ncpol{5} }, \\
	\left( \set{ \set{1, 2, 5}, \set{3, 4} }, 101001 \right) 
	\quad & \mapsto \quad
	\pncp{ \nccup{1}{5} \nccup[0.9]{3}{4} \ncstick{2} \ncpol{1} \ncpol{3} \ncpolfin{5} }.
\end{align*}
We let $\PNCP(n)$ denote the set of polarised noncrossing partitions of degree $n$; clearly, $\size{\PNCP(n)} = 2^{n+1}C_n$.

We will induce an operadic structure on polarised noncrossing partitions from an independent operadic structure on the polarities, as follows.

\begin{dfn}[Bit operad]
	The \emph{bit operad} $\Bop$ has sets of operations $\Bop(0) = \varnothing$ and $\Bop(n) \eqdef (\FF)^n$ for $n > 0$, whose elements we write as strings $b_1 \cdots b_n$, with $0 \in \Bop(1)$ as unit operation, and partial composition defined by
	\begin{equation*}
		(b_1 \cdots b_n) \pcirc_i (c_1 \cdots c_m) \eqdef b_1 \cdots b_{i-1}(b_i \oplus c_1)\cdots(b_i \oplus c_m)b_{i+1}\cdots b_n
	\end{equation*}
for all $b_1 \cdots b_n \in \Bop(n)$, $c_1 \cdots c_m \in \Bop(m)$, and all $i \in \set{1, \ldots, n}$.
\end{dfn}

\begin{dfn}[Level of a bit string]
	Let $b_1 \cdots b_n$ be a string of $n$ bits.
	The \emph{level} of $b_1 \cdots b_n$ is the natural number
	\begin{equation*}
		\levB(b_1 \cdots b_n) \eqdef \max \left( \set{i \mid b_i = 1} \cup \set{0} \right).
	\end{equation*}
\end{dfn}


\noindent
Recall that, given two operads $\oper{P}$ and $\oper{Q}$, we can form a \emph{product operad} $\oper{P} \times \oper{Q}$ whose set of $n$\nbd ary operations is $\oper{P}(n) \times \oper{Q}(n)$, and unit and partial composition operations are induced by those of $\oper{P}$ and $\oper{Q}$ separately in each factor.

\begin{dfn}[Polarised gap-insertion operad]
	The \emph{polarised gap-insertion operad} is the product $\PNCPop$ of the gap-insertion operad $\NCPop$ with the bit operad $\Bop$.
	For each $n \in \NN$, we identify the set $\PNCPop(n+1) = \NCP(n) \times (\FF)^{n+1}$ with the set $\PNCP(n)$ of polarised noncrossing partitions of degree $n$.
\end{dfn}

\begin{rmk} 
	The sequence of injective maps $\NCP(n) \incl \PNCP(n)$ given by $p \mapsto (p, 0\cdots 0)$ determines an injective operad morphism from $\NCPop$ to $\PNCPop$, exhibiting the gap-insertion operad as a sub-operad of the polarised gap-insertion operad.
\end{rmk}

\noindent
We will not be particularly concerned with $\PNCPop$ \emph{as an operad}; the main utility of this structure is that it allows us to state a unique decomposition result generalising Proposition \ref{prop:ncp_unique_factorisation}.

\begin{dfn}[Level of a polarised noncrossing partition]
	Let $(p, b_\bullet)$ be a polarised noncrossing partition. The \emph{level} of $(p, b_\bullet)$ is the natural number
\begin{equation*}
	\lev(p, b_\bullet) \eqdef \max \set{\levNCP(p), \levB(b_\bullet)}.
\end{equation*}
\end{dfn}

\noindent
For each $n > 0$ and $b \in \set{0, 1}$, we let $\coarse_n^b \eqdef (\coarse_n, b0\cdots 0)$.
Moreover, we let $\coarse_0^b \eqdef (\set{}, b)$; in particular, $\coarse_0^0$ is the unit of $\PNCPop$.

\begin{lem} \label{lem:final_polarised_block}
	Let $(p, b_\bullet)$ be a polarised noncrossing partition and suppose that $i \eqdef \lev(p, b_\bullet) > 0$.
	Then there exists a unique triple of a polarised noncrossing partition $(p', b'_\bullet)$, $m \in \NN$, and $c \in \set{0, 1}$, such that
	\begin{enumerate}
		\item $(p, b_\bullet) = (p', b'_\bullet) \pcirc_i \coarse_m^c$,
		\item $\lev(p', b'_\bullet) < \lev(p, b_\bullet)$.
	\end{enumerate}
	Moreover, $c = b_i$.
\end{lem}
\begin{proof}
	In what follows, $b'_\bullet$ is a string of the appropriate length $\geq i$ depending on context, defined by $b'_j \eqdef b_j$ for all $j < i$ and $b'_i \eqdef 0$ for all $j \geq i$; by construction, $\levB(b'_\bullet) < i$.
	First, suppose that $i = \levB(b_\bullet) > \levNCP(p)$.
	Then $b_i = 1$, $(p, b_\bullet) = (p, b'_\bullet) \pcirc_i \coarse_0^1$ with $\lev(p, b'_\bullet) < i$, and the factorisation is unique, because any factorisation of the form $(p', b''_\bullet) \pcirc_i \coarse_m^c$ with $m > 0$ would imply $\levNCP(p) = i$, a contradiction.
	Next, suppose that $i = \levNCP(p) \geq \levB(b_\bullet)$.
	Then Lemma \ref{lem:final_block_factorisation} provides a factorisation $p = p' \pcirc_i \coarse_m$, and $(p, b_\bullet) = (p', b'_\bullet) \pcirc_i \coarse_m^{b_i}$ with $\lev(p', b'_\bullet) < i$, proving existence.
	Suppose that $(p'', b''_\bullet) \pcirc_i \coarse_{m'}^c$ is another factorisation satisfying the same conditions.
	If $m' > 0$, then $p = p'' \pcirc_i \coarse_{m'}$, so $p'' = p'$ and $m' = m$ by Lemma \ref{lem:final_block_factorisation}.
	Moreover, $b''_j = b_j$ for all $j \neq i$, while $b''_i = b_i \oplus c$, so if $b_i = c$, then $b''_\bullet = b'_\bullet$, and we recover the original factorisation; if $b_i \neq c$, then $\levB(b''_\bullet) = i$, contradicting $\lev(p'', b''_\bullet) < i$.
	Finally, if $m' = 0$, then $\levNCP(p'') = \levNCP(p) = i$, again contradicting $\lev(p'', b''_\bullet) < i$.
\end{proof}

\begin{prop} \label{prop:unique_factorisation}
	Let $(p, b_\bullet)$ be a polarised noncrossing partition.
	Then there exists a unique factorisation of the form
	\begin{equation*}
		(p, b_\bullet) = (\ldots(\coarse_0^0 \pcirc_{i_1} \coarse^{c_1}_{m_1}) \pcirc_{i_2} \coarse^{c_2}_{m_2} \ldots) \pcirc_{i_k} \coarse^{c_k}_{m_k}
	\end{equation*}
	with $i_1 < \ldots < i_k$, and, for all $j \in \set{1, \ldots, k}$, either $m_j > 0$ or $c_j = 1$.
	Moreover, $\lev(p, b_\bullet) = \max \set{0, i_1, \ldots, i_k}$, and $c_j = b_{i_j}$ for all $j \in \set{1, \ldots, k}$.
\end{prop}
\begin{proof}
	Existence follows by well-founded recursion from Lemma \ref{lem:final_polarised_block}, with $\ell(-)$ as the strictly decreasing index.
	Uniqueness also follows from the observation that each subfactor ending in $\pcirc_{i_j} \coarse_{m_j}^{c_j}$ has level $i_j$.
\end{proof}

\begin{dfn}[Rightmost decomposition]
	The \emph{rightmost decomposition} of a polarised noncrossing partition $(p, b_\bullet)$ is its factorisation granted by Proposition \ref{prop:unique_factorisation}.
\end{dfn}

\noindent
We may thus uniquely encode a polarised noncrossing partition as a sequence $(i_j, m_j, b_j)_{j=1}^k$ of triples of numbers $i_j \in \NN_{>0}$, $m_j \in \NN$, $b_j \in \set{0, 1}$, subject to the constraints 
\begin{enumerate}
	\item $i_1 < \ldots < i_k$,
	\item for all $j \in \set{1, \ldots, k}$, $i_j \leq 1 + \sum_{\ell=1}^{j-1} m_\ell$,
	\item for all $j \in \set{1, \ldots, k}$, either $m_j > 0$ or $b_j = 1$.
\end{enumerate}
Notice that these force $i_1 = 1$.
We call $k$ the \emph{length} of the rightmost decomposition; observe that 
\begin{enumerate}
	\item the $\coarse_m^b$ are precisely the polarised noncrossing partitions whose rightmost decomposition has length at most 1,
	\item every other polarised noncrossing partition has a rightmost decomposition of the form $(p, b_\bullet) \pcirc_i \coarse_m^c$ where $(p, b_\bullet)$ has a rightmost decomposition of strictly lower length,
\end{enumerate}
and this will allow us to prove facts by induction on the length of the rightmost decomposition.

\begin{rmk}
	Given the representation $(i_j, m_j, b_j)_{j=1}^k$ of a polarised noncrossing partition, one can recover its degree as $\sum_{j=1}^k m_j$ and its level as $\max \set{0, i_1, \ldots, i_k}$.
\end{rmk}

\begin{exm}
The rightmost decomposition of $\pncp{ \nccup{1}{5} \nccup[0.9]{3}{4} \ncstick{2} \ncpol{2} \ncpol{4} \ncpol{5} }$ is
\begin{align*}
	((((\coarse_0^0 \pcirc_1 \coarse_3^0) & \pcirc_2 \coarse_0^1) \pcirc_3 \coarse_2^0) \pcirc_4 \coarse_0^1) \pcirc_5 \coarse_0^1 \\
					      & = ((1, 3, 0), (2, 0, 1), (3, 2, 0), (4, 0, 1), (5, 0, 1)),
\end{align*}
while the rightmost decomposition of $\pncp{ \nccup{1}{5} \nccup[0.9]{3}{4} \ncstick{2} \ncpol{1} \ncpol{3} \ncpolfin{5} }$ is
\[
	((\coarse_0^0 \pcirc_1 \coarse_3^1) \pcirc_3 \coarse_2^1) \pcirc_6 \coarse_0^1 = 
	((1, 3, 1), (3, 2, 1), (6, 0, 1)).
\]
\end{exm}

\section{The augmented chain complex} \label{sec:chain}

\noindent
The goal of this section is to define an augmented chain complex related to the operad of polarised noncrossing partitions, and study its properties.

\begin{dfn}[Augmented chain complex]
An \emph{augmented chain complex} $C_\bullet$ is a chain complex of abelian groups in non-negative degree
\[\begin{tikzcd}
\ldots & {C_n} & {C_{n-1}} & \ldots & {C_{1}} & {C_{0}}
\arrow["\der", from=1-1, to=1-2]
\arrow["\der", from=1-2, to=1-3]
\arrow["\der", from=1-3, to=1-4]
\arrow["\der", from=1-4, to=1-5]
\arrow["\der", from=1-5, to=1-6]
\end{tikzcd}\]
together with a homomorphism $\aug\colon C_{0} \to \ZZ$ satisfying $\aug\der = 0$.
\end{dfn}

\noindent
We let $\chaug$ denote the category of augmented chain complexes and homomorphisms (that is, degree-wise homomorphisms of abelian groups which commute with the differentials and augmentation).

Given a set $B$, we let $\ZZ B$ denote the free abelian group on $B$.

\begin{dfn}[The augmented chain complex $E_\bullet$]
We define an augmented chain complex $E_\bullet$ as follows.
We let 
\[
	E_0 \eqdef \ZZ \set{s, t}, \quad \quad E_{n+1} \eqdef \ZZ \left(\PNCP(n)\right)
\]
for each $n \in \NN$.
Since the abelian groups are free in each degree, it suffices to define $\der$ and $\aug$ on generators.
We let $\aug s = \aug t  = 1$.
Next, for each $n \in \NN$ and $(p, b_\bullet) \in \PNCP(n)$, we define $\der(p, b_\bullet)$ by induction on the length of the rightmost decomposition of $(p, b_\bullet)$. 
First of all, if $(p, b_\bullet) = \coarse_m^c$ for some $m \in \NN$ and $c \in \set{0, 1}$, we let, inductively on $m$,
\begin{align*}
	\der \coarse_0^c & \eqdef (-)^c(t - s), \\
	\der \coarse_{m+1}^c & \eqdef \coarse_m^{c \oplus 1} + (-)^m \coarse_m^c.
\end{align*}
Next, suppose that the rightmost decomposition of $(p, b_\bullet)$ is of the form $(p', b'_\bullet) \pcirc_i \coarse_m^c$ with $p'$ of degree $k > 0$.
We distinguish two cases.
\begin{itemize}
	\item \emph{Case $i = k+1$}. Then
	\begin{align*}
		\der\left( (p', b'_\bullet) \pcirc_i \coarse_0^1 \right) & 
			\eqdef - \der (p', b'_\bullet), \\
		\der\left( (p', b'_\bullet) \pcirc_i \coarse_{m+1}^c \right) & 
			\eqdef (p', b'_\bullet) \pcirc_i \der \coarse_{m+1}^c.
	\end{align*}
	\item  \emph{Case $i < k+1$}. Then
	\begin{align*}
		\der\left( (p', b'_\bullet) \pcirc_i \coarse_0^1 \right) & 
			\eqdef \der (p', b'_\bullet) \pcirc_i \coarse_0^1, \\
		\der\left( (p', b'_\bullet) \pcirc_i \coarse_{m+1}^c \right) & 
			\eqdef \der(p', b'_\bullet) \pcirc_i \coarse_{m+1}^c + (-)^{k+1-i}(p', b'_\bullet) \pcirc_i \der \coarse_{m+1}^c.
	\end{align*}
\end{itemize}
Wherever necessary, $- \pcirc_i -$ is extended bilinearly in the evident way.
\end{dfn}

\begin{exm}
	Let us compute $\der \pncp{ \nccup{1}{5} \nccup[0.9]{3}{4} \ncstick{2} \ncpol{1} \ncpol{3} \ncpolfin{5} }$ according to the definition.
	First of all, we have $\pncp{ \nccup{1}{5} \nccup[0.9]{3}{4} \ncstick{2} \ncpol{1} \ncpol{3} \ncpolfin{5} } = \pncp{ \nccup{1}{5} \nccup[0.9]{3}{4} \ncstick{2} \ncpol{1} \ncpol{3} } \pcirc_6 \coarse_0^1$, so
	\[
		\der \pncp{ \nccup{1}{5} \nccup[0.9]{3}{4} \ncstick{2} \ncpol{1} \ncpol{3} \ncpolfin{5} } = - \der \pncp{ \nccup{1}{5} \nccup[0.9]{3}{4} \ncstick{2} \ncpol{1} \ncpol{3} }.
	\]
	Next, $\pncp{\nccup{1}{5} \nccup[0.9]{3}{4} \ncstick{2} \ncpol{1} \ncpol{3} } = \pncp{ \nccup{1}{3} \ncstick{2} \ncpol{1} } \pcirc_3 \coarse_2^1$, so
	\[
		\der \pncp{\nccup{1}{5} \nccup[0.9]{3}{4} \ncstick{2} \ncpol{1} \ncpol{3} } =  \der \pncp{ \nccup{1}{3} \ncstick{2} \ncpol{1} } \pcirc_3 \coarse_2^1 + (-)^{4-3} \pncp{ \nccup{1}{3} \ncstick{2} \ncpol{1} } \pcirc_3 \der \coarse_2^1 
		= \der \coarse_3^1 \pcirc_3 \coarse_2^1 - \coarse_3^1 \pcirc_3 \der \coarse_2^1.
	\]
	Finally, we have
	\[
		\der \coarse_3^1 = \coarse_2^0 + \coarse_2^1, \quad \quad \der \coarse_2^1 = \coarse_1^0 - \coarse_1^1.
	\]
	Putting everything together, we have
\begin{align*}
	\der \pncp{ \nccup{1}{5} \nccup[0.9]{3}{4} \ncstick{2} \ncpol{1} \ncpol{3} \ncpolfin{5} } & =
	- \left( (\coarse_2^0 + \coarse_2^1) \pcirc_3 \coarse_2^1  - \coarse_3^1 \pcirc_3 (\coarse_1^0 - \coarse_1^1) \right) = \\
		& = - \coarse_2^0 \pcirc_3 \coarse_2^1 - \coarse_2^1 \pcirc_3 \coarse_2^1 + \coarse_3^1 \pcirc_3 \coarse_1^0 - \coarse_3^1 \pcirc_3 \coarse_1^1 = \\
		& = - \pncp{ \nccup{1}{2} \nccup{3}{4} \ncpol{3} } - 
		\pncp{ \nccup{1}{2} \nccup{3}{4} \ncpol{1} \ncpol{3} } +
		\pncp{ \nccup{1}{4} \ncstick{2} \ncstick[0.9]{3} \ncpol{1} } -
		\pncp{ \nccup{1}{4} \ncstick{2} \ncstick[0.9]{3} \ncpol{1} \ncpol{3} }.
\end{align*}
\end{exm}

\begin{prop} \label{prop:chain_complex_well_defined}
	The augmented chain complex $E_\bullet$ is well-defined.
\end{prop}
\begin{proof}
	The definition of $\der$ is complete and unambiguous by uniqueness of rightmost decompositions, so it suffices to show that $\aug\der = 0$ and $\der\der = 0$.
	By freeness of the groups, it suffices to check this on generators.
	For each $c \in \set{0, 1}$, we have
	\[
		\aug\der \coarse_0^c = (-)^c(\aug t - \aug s) = (-)^c(1 - 1) = 0.
	\]
	For generators $(p, b_\bullet)$ in higher degrees, we proceed by induction on the length of the rightmost decomposition of $(p, b_\bullet)$.
	Suppose $(p, b_\bullet) = \coarse_m^c$ for some $m > 0$ and $c \in \set{0, 1}$. 
	If $m = 1$, we have
	\[
		\der\der \coarse_1^c = \der\coarse_0^{c \oplus 1} + \der\coarse_0^c = (-)^{c+1}(t-s) + (-)^c(t-s) = 0.
	\]
	If $m > 1$, we have
	\begin{align*}
		\der\der \coarse_m^c & = \der\coarse_{m-1}^{c \oplus 1} + (-)^{m-1}\der\coarse_{m-1}^c = \\
		& = \coarse_{m-2}^c + (-)^{m-2}\coarse_{m-2}^{c \oplus 1} + (-)^{m-1}\left(\coarse_{m-2}^{c \oplus 1} + (-)^{m-2}\coarse_{m-2}^c\right) = \\
		& = \left( 1 + (-)^{2m-3} \right) \coarse_{m-2}^c + \left( (-)^{m-2} + (-)^{m-1} \right) \coarse_{m-2}^{c \oplus 1} = 0.
	\end{align*}
	Otherwise, $(p, b_\bullet)$ has a rightmost decomposition of the form $(p', b'_\bullet) \pcirc_i \coarse_m^c$ with $p'$ of degree $k > 0$; we may assume inductively that $\der \der (p', b'_\bullet) = 0$.
	We proceed by a case distinction.
	If $i = k+1$, $m = 0$, $c = 1$, we have $\der\der (p, b_\bullet) = - \der\der (p', b'_\bullet) = 0$.
	If $i = k+1$ and $m > 0$, we have 
	\[
		\der\der (p, b_\bullet) = (p', b'_\bullet) \pcirc_i \der \der \coarse_{m}^c = 0
	\]
	by bilinearity and the first part of the proof.
	Next, suppose that $i = k$.
	If $m = 0$, $c = 1$, then
	\[
		\der\der (p, b_\bullet) = \der \left(\der(p', b'_\bullet) \pcirc_i \coarse_0^1\right) = - \der\der(p', b'_\bullet) = 0,
	\]
	if $m = 1$, then
	\begin{align*}
		\der\der (p, b_\bullet) & = \der \left( \der(p', b'_\bullet) \pcirc_i \coarse_1^c - (p', b'_\bullet) \pcirc_i \der \coarse_1^c \right) = \\
		& = \der (p', b'_\bullet) \pcirc_i \der \coarse_1^c - \der \left( (p', b'_\bullet) + (p', b'_\bullet) \pcirc_i \coarse_0^1 \right) = \\
		& = \der (p', b'_\bullet) + \der (p', b'_\bullet) \pcirc_i \coarse_0^1 - \der (p', b'_\bullet) - \der (p', b'_\bullet) \pcirc_i \coarse_0^1 = 0,
	\end{align*}
	and if $m > 1$, then
	\begin{align*}
		\der\der (p, b_\bullet) & = \der \left( \der(p', b'_\bullet) \pcirc_i \coarse_m^c - (p', b'_\bullet) \pcirc_i \der \coarse_m^c \right) = \\
		& = \der (p', b'_\bullet) \pcirc_i \der \coarse_m^c - \der (p', b'_\bullet) \pcirc_i \der \coarse_m^c + (p', b'_\bullet) \pcirc_i \der \der \coarse_m^c = 0. 
	\end{align*}
	Finally, suppose that $i < k$.
	If $m = 0$, $c = 1$, then
	\[
		\der\der (p, b_\bullet) = \der \der (p', b'_\bullet) \pcirc_i \coarse_0^1 = 0,
	\]
	if $m = 1$, then
	\begin{align*}
		\der \der (p, b_\bullet) & = \der \left( \der(p', b'_\bullet) \pcirc_i \coarse_1^c + (-)^{k+1-i}\left( (p', b'_\bullet) + (p', b'_\bullet) \pcirc_i \coarse_0^1 \right) \right) = \\
			& = (-)^{k-i} \der(p', b'_\bullet) \pcirc_i \der \coarse_1^c + (-)^{k+1-i} \der \left( (p', b'_\bullet) + (p', b'_\bullet) \pcirc_i \coarse_0^1 \right) = \\
			& = \left((-)^{k-i} + (-)^{k+1-i}\right)
			\left( \der(p', b'_\bullet) + \der(p', b'_\bullet) \pcirc_i \coarse_0^1 \right) = 0,
	\end{align*}
	and if $m > 1$, then $\der \der (p, b_\bullet)$ is equal to
	\[
		\der\der (p', b'_\bullet) \pcirc_i \coarse_m^c + \left((-)^{k-i} + (-)^{k+1-i}\right) \der(p', b'_\bullet) \pcirc_i \der \coarse_m^c + (p', b'_\bullet) \pcirc_i \der \der \coarse_m^c = 0.
	\]
	This concludes the case distinction and the proof.
\end{proof}

\begin{rmk}
	This chain complex seems to be unrelated to the one arising when studying the homology of the \emph{lattice} of noncrossing partitions \cite{edelman1980chain}. 
\end{rmk}

\noindent
Next, we show that $E_\bullet$ is ``self-dual'' in the sense that it has a non-trivial involutive automorphism swapping $s$ and $t$ in degree 0.

\begin{prop} \label{prop:chain_selfdual}
	Let $\tau\colon E_\bullet \to E_\bullet$ be defined on generators by
	\begin{align*}
		& s \mapsto t,\, t \mapsto s && \text{in degree $0$,} \\
		& (p, b_1 \cdots b_{n+1}) \mapsto (p, (b_1 \oplus 1)b_2 \cdots b_{n+1}) && \text{in degree $> 0$}.
	\end{align*}
	Then $\tau$ is an involutive automorphism of augmented chain complexes.
\end{prop}
\begin{proof}
	First of all, $\tau$ is evidently a degreewise involutive bijection on generators and it commutes with the augmentation, so it suffices to show that $\der \tau = \tau \der$ in every degree.
	Let $(p, b_\bullet)$ be a polarised noncrossing partition; we proceed by induction on the length of its rightmost decomposition.
	Observe that
	\[
		\tau(p, b_\bullet) = (p, b_\bullet) \pcirc_1 \coarse_0^1,
	\]
	but this is not a rightmost decomposition unless $(p, b_\bullet)$ is $\coarse_0^0$.
	If $(p, b_\bullet) = \coarse_m^c$ for some $m \in \NN$ and $c \in \set{0, 1}$, then $\tau(p, b_\bullet) = \coarse_m^{c\oplus 1}$, so
	\begin{align*}
		\der \tau \coarse_0^c & = \der\coarse_0^{c\oplus 1} = (-)^{c+1}(t - s) = (-)^c(s - t) = \tau \der\coarse_0^c, \\
		\der \tau \coarse_{m+1}^c & = \der\coarse_{m+1}^{c\oplus 1} = \coarse_{m}^{c} + (-)^m \coarse_m^{c \oplus 1} = \tau \der\coarse_{m+1}^c.
	\end{align*}
	Otherwise, if $(p, b_\bullet) = (p', b'_\bullet) \pcirc_i \coarse_m^c$ with $i > 1$ and $p'$ of degree $k > 0$, then 
	\begin{align*}
		\tau(p, b_\bullet) 
		& = ((p', b'_\bullet) \pcirc_i \coarse_m^c) \pcirc_1 \coarse_0^1 = \\
		& = ((p', b'_\bullet) \pcirc_1 \coarse_0^1) \pcirc_i \coarse_m^c = \tau(p', b'_\bullet) \pcirc_i \coarse_m^c
	\end{align*}
	by parallel associativity.
	It follows that, if $i < k +1$, then if $m = 0$, $c = 1$
	\[
		\der \tau \left( (p', b'_\bullet) \pcirc_i \coarse_0^1 \right) = \der \tau (p', b'_\bullet) \pcirc_i \coarse_0^1 = \tau \der (p', b'_\bullet) \pcirc_i \coarse_0^1 = \tau \der \left( (p', b'_\bullet) \pcirc_i \coarse_0^1 \right), 
	\]
	and if $m > 0$
	\begin{align*}
		\der \tau \left( (p', b'_\bullet) \pcirc_i \coarse_m^c \right) & = \der\tau(p', b'_\bullet) \pcirc_i \coarse_m^c + (-)^{k+1-i} \tau(p', b'_\bullet) \pcirc_i \der \coarse_m^c = \\
									       & = \tau\der(p', b'_\bullet) \pcirc_i \coarse_m^c + (-)^{k+1-i} \tau (p', b'_\bullet) \pcirc_i \der\coarse_m^c = \\
									       & = \tau \der \left( (p', b'_\bullet) \pcirc_i \coarse_m^c \right).
	\end{align*}
	The case $i = k+1$ is analogous and even easier.
\end{proof}

\begin{dfn}[Level filtration]
	Let $k \in \NN$.
	We let $\filtr{k}{E}_0 \eqdef E_0$ and
	\begin{align*}
		\filtr{k}{\PNCP}(n) & \eqdef \set{ (p, b_\bullet) \in \PNCP(n) \mid \lev(p, b_\bullet) \leq k }, \\
		\filtr{k}{E}_{n+1} & \eqdef \ZZ \filtr{k}{\PNCP}(n)
	\end{align*}
	for each $n \in \NN$.
	This determines a chain of degreewise injective homomorphisms of abelian groups
\[\begin{tikzcd}
	{\filtr{0}{E}_n} & {\filtr{1}{E}_n} & \ldots & {\filtr{k}{E}_n} & \ldots
	\arrow[hook, from=1-1, to=1-2]
	\arrow[hook, from=1-2, to=1-3]
	\arrow[hook, from=1-3, to=1-4]
	\arrow[hook, from=1-4, to=1-5]
\end{tikzcd}\]
	which, for each $n \in \NN$, stabilises at $k = n$, that is, $\filtr{k}{E}_n = E_n$ for all $k \geq n$.
	We call this the \emph{level filtration} of $E_\bullet$.
\end{dfn}

\begin{prop} \label{prop:filtration_is_well_formed}
	Let $k \in \NN$. 
	Then $\filtr{k}{E}_\bullet$ determines a subcomplex of $E_\bullet$.
	Consequently, the level filtration is a filtration of augmented chain complexes
\[\begin{tikzcd}
	{\filtr{0}{E}_\bullet} & {\filtr{1}{E}_\bullet} & \ldots & {\filtr{k}{E}_\bullet} & \ldots
	\arrow[hook, from=1-1, to=1-2]
	\arrow[hook, from=1-2, to=1-3]
	\arrow[hook, from=1-3, to=1-4]
	\arrow[hook, from=1-4, to=1-5]
\end{tikzcd}\]
	whose colimit is $E_\bullet$.
\end{prop}
\begin{proof}
	For the generators $s, t$ of degree $0$, there is nothing to prove as they are included in each level.
	By Proposition \ref{prop:unique_factorisation}, the level of a generator in higher degree is determined by the final factor in its rightmost decomposition.
	The Leibniz-rule definition of $\der$ then guarantees that $\der$ never raises the level: each term of $\der(p, b_\bullet)$ either keeps the final-factor gap, or drops it, leading to a factor with a strictly lower level.
\end{proof}

\begin{prop} 
	Let $k > 0$.
	Then $\tau\colon E_\bullet \to E_\bullet$ restricts to an involutive automorphism on $\filtr{k}{E}_\bullet$.
\end{prop}
\begin{proof}
	The only generator whose level is strictly raised by $\tau$ is $\coarse_0^0$.
	Since both $\coarse_0^0$ and its dual $\coarse_0^1$ appear in each level $k > 0$, we conclude.
\end{proof}

\noindent
Recall that the \emph{Catalan triangle numbers} are defined by
\[
	T(n, k) \eqdef \binom{n+k}{k} - \binom{n+k}{k-1}, \quad 0 \leq k \leq n
\]
and that they satisfy the vertical summation identities
\begin{equation} \label{eq:vertical_summation}
	T(n, k) = \sum_{i=k}^n T(i, k-1), \quad 1 \leq k \leq n.
\end{equation}
In particular, $T(n, n) = C_n$, the $n$-th Catalan number, while $T(n, 0) = 1$.
We also let $\delta_n$ be $1$ if $n = 0$ and $0$ if $n > 0$.

\begin{prop} \label{prop:enumeration_of_Enk}
	Let $n > 0$, $k \in \NN$, and let 
	\[
		E(n, k) \eqdef \size{ \filtr{k}{\PNCP}(n-1) }.
	\]
	Then
	\[
		E(n, k) = \begin{cases}
			\delta_{n-1} & \text{if $k = 0$,} \\
			2^{\min(n, k)} T(n-1, \min(n, k)-1) & \text{if $k > 0$.}
		\end{cases}
	\]
\end{prop}
\begin{proof}
	The only polarised noncrossing partition of level 0 is $\coarse_0^0$, which has degree 0, so $E(1, 0) = 1$ and $E(n, 0) = 0$ for all $n > 1$.
	Suppose $k > 0$.
	If $k \leq n$, by Proposition \ref{prop:unique_factorisation}, the polarised noncrossing partitions of level $k$ are those whose rightmost decomposition is of the form $(p, b_\bullet) \pcirc_k \coarse_m^c$, where $(p, b_\bullet)$ has strictly lower level.
	Thus, a polarised noncrossing partition of degree $n-1$ and level \emph{at most} $k$ is either
	\begin{itemize}
		\item a polarised noncrossing partition of strictly lower level, of which there are $E(n, k-1)$, or
		\item $(p, b_\bullet) \pcirc_k \coarse_0^1$, where $(p, b_\bullet)$ has strictly lower level and degree $n-1$, of which there are also $E(n, k-1)$,
		\item $(p, b_\bullet) \pcirc_k \coarse_{n-i}^c$, where $(p, b_\bullet)$ has strictly lower level and degree $i-1$ with $k \leq i < n$, of which there are $2 \sum_{i=k}^{n-1} E(i, k-1)$.
	\end{itemize}
	Overall, we have $E(n, k) = 2\sum_{i=k}^n E(i, k-1)$.
	Now, inductively on $0 < k \leq n$, we have, for $k = 1$,
	\[
		E(n, 1) = 2\sum_{i=1}^n E(i, 0) = 2E(1, 0) = 2 = 2^1 T(n-1, 0),
	\]
	and for $k > 1$
	\begin{align*}
		E(n, k) & = 2\sum_{i=k}^n E(i, k-1) = 2\sum_{i=k}^n 2^{k-1} T(i-1, k-2) = \\
			& = 2^k \sum_{j={k-1}}^{n-1} T(j, k-2) = 2^k T(n-1, k-1)
	\end{align*}
	using the vertical summation identity (\ref{eq:vertical_summation}).
	Finally, for $k > n$, every polarised noncrossing partition of degree $n-1$ has level $< k$, so
	\[
		E(n, k) = E(n, n) = 2^n T(n-1, n-1) = 2^{\min(n,k)} T(n-1, \min(n,k) - 1). \qedhere
	\]
\end{proof}

\noindent
The rest of this section will be devoted to the proof that $E_\bullet$, as well as each of the levels $\filtr{k}{E}_\bullet$, is acyclic.
First, we recall some basic notions of simplicial homotopy theory \cite{goerss1999simplicial}.

Given a simplicial set $K$ and $n \in \NN$, we let $K[n]$ denote its set of $n$-simplices, and for each $i \in \set{0, \ldots, n}$, we let $d_i\colon K[n] \to K[n-1]$ denote the $i$\nbd th face map and $s_i\colon K[n] \to K[n+1]$ the $i$\nbd th degeneracy map.
We let $\Nd K[n]$ denote the set of \emph{nondegenerate} $n$\nbd simplices in $K$, that is, the complement of $\bigcup_{i=0}^{n-1} s_i(K[n-1])$ in $K[n]$.
One can associate to $K$ a chain complex $\NSimp K_\bullet$, the \emph{normalised simplicial chain complex} of $K$, which in degree $n$ is defined by
\[
	\NSimp K_n \eqdef \ZZ \left(\Nd K[n]\right),
\]
with differential maps defined, for each $n > 0$ and $x \in \Nd K[n]$, by
\[
	\der\colon x \mapsto \sum_{i \mid d_i x \in \Nd K[n-1]} (-)^i d_i x.
\]
This can be given an augmentation defined by $\aug\colon x \mapsto 1$ for each $x \in K[0]$, which determines a functor
\[
	\NSimp\colon \sSet \to \chaug.
\]
The homology of $\NSimp K_\bullet$ is the homology of $K$; in particular, if $K$ is contractible, then $\NSimp K_\bullet$ is acyclic.

\begin{lem} \label{lem:E0_acyclic}
	$\filtr{0}{E}_\bullet$ is acyclic.
\end{lem}
\begin{proof}
	The only generators of $\filtr{0}{E}_\bullet$ are $\set{s, t}$ in degree $0$ and $\coarse_0^0$ in degree $1$, with $\der \coarse_0^0 = t - s$.
	This is easily seen to be isomorphic to the normalised simplicial chain complex of the 1-simplex, which is acyclic since it computes the homology of the interval.
\end{proof}

\noindent
Given a set $S$, there is a simplicial set $\cosk_0 S$ whose $n$-simplices are strings $x_0 \cdots x_n$ of elements of $S$, and face and degeneracy maps are defined by
\begin{align*}
	d_i(x_0 \cdots x_n) & \eqdef x_0\cdots x_{i-1}x_{i+1}\cdots x_n, \quad \quad \\
	s_i(x_0 \cdots x_n) & \eqdef x_0\cdots x_{i-1}x_i x_i x_{i+1}\cdots x_n
\end{align*}
for each $i \in \set{0, \ldots, n}$.
When $S$ has at least one element $x$, then $\cosk_0 S$ is contractible; this is most easily established by observing that the family of maps $x_0\cdots x_n \mapsto x x_0 \cdots x_n$ determines an \emph{extra degeneracy} on $K$ \cite[Section 5.5]{goerss1999simplicial}.
Moreover, the nondegenerate simplices of $\cosk_0 S$ are readily established to be precisely the strings with no consecutive repeated entries.

\begin{lem} \label{lem:characterisation_of_E1}
	$\filtr{1}{E}_\bullet$ is isomorphic to $\NSimp (\cosk_0 \set{s, t})_\bullet$.
	In particular, it is acyclic.
\end{lem}
\begin{proof}
	For each $n \in \NN$, the only nondegenerate $n$\nbd simplices in $\cosk_0 \set{s, t}$ are the alternating strings of $s$ and $t$ of length $n+1$; that is,
	\[
		\underbrace{stst\cdots}_{n+1}, \quad \quad \underbrace{tsts\cdots}_{n+1}.
	\]
	which end with the same character as their first when $n$ is even, and with the other character when $n$ is odd.
	Moreover, for $n > 0$, their only nondegenerate faces are $d_0$ and $d_n$, since removing any inner entry produces a string with a repetition.
	Thus, the differential maps in $\NSimp (\cosk_0 \set{s, t})_\bullet$ are defined by
	\begin{align*}
		\der(\underbrace{stst\cdots}_{n+1}) & = \underbrace{tsts\cdots}_n + (-)^n\underbrace{stst\cdots}_n, \\
		\der(\underbrace{tsts\cdots}_{n+1}) & = \underbrace{stst\cdots}_n + (-)^n\underbrace{tsts\cdots }_n.
	\end{align*}
	On the other hand, $\filtr{1}{E}_\bullet$ is generated by $\set{s, t}$ in degree $0$, and by $\set{\coarse_n^0, \coarse_n^1}$ in degree $n + 1$ for each $n \in \NN$.
	The groups in degree $0$ are literally identical, and for each $n \in \NN$, by inspection, the assignment
	\[
		\coarse_n^0 \mapsto \underbrace{stst\cdots}_{n+2}, \quad \quad \coarse_n^1 \mapsto \underbrace{tsts\cdots}_{n+2}
	\]
	is bijective on generators and compatible with the differentials in degree $n+1$, hence determines an isomorphism of augmented chain complexes.
\end{proof}

\begin{thm} \label{thm:acyclicity}
	For each $k \in \NN$, $\filtr{k}{E}_\bullet$ is acyclic.
	Consequently, $E_\bullet$ is acyclic.
\end{thm}
\begin{proof}
	By Lemma \ref{lem:E0_acyclic}, $\filtr{0}{E}_\bullet$ is acyclic, thus by a standard argument---using exact sequences and the commutation of homology with sequential colimits---it suffices to show that, for each $k > 0$, the quotient complex $\filtr{k}{Q}_\bullet \eqdef \filtr{k}{E}_\bullet / \filtr{k-1}{E}_\bullet$ is acyclic.
	For $k = 1$, by Lemma \ref{lem:characterisation_of_E1} combined with Lemma \ref{lem:E0_acyclic}, we have that $\filtr{1}{Q}_\bullet$ is the quotient of an acyclic complex by an acyclic subcomplex, so it is acyclic.
	We will let $\derbar$ denote its differential.

	For $k > 1$, by Proposition \ref{prop:unique_factorisation}, each generator of $\filtr{k}{Q}_\bullet$ is uniquely of the form $(p, b_\bullet) \pcirc_k \coarse_m^c$ where $m > 0$ or $c = 1$, and $(p, b_\bullet)$ has level strictly lower than $k$; observe that the $\coarse_m^c$ appearing in this expression correspond bijectively to the generators of $\filtr{1}{Q}_\bullet$.
	For each $j \in \NN$, restricting $\filtr{k}{Q}_\bullet$ to generators $(p, b_\bullet) \pcirc_k \coarse_m^c$ with $(p, b_\bullet)$ of degree at most $j$ determines a subcomplex $\filtr{k, j}{Q}_\bullet$, so we have a filtration of chain complexes
\begin{equation} \label{eq:base_order_filtration}
\begin{tikzcd}
	{\filtr{k,0}{Q}_\bullet} & {\filtr{k,1}{Q}_\bullet} & \ldots & {\filtr{k,j}{Q}_\bullet} & \ldots
	\arrow[hook, from=1-1, to=1-2]
	\arrow[hook, from=1-2, to=1-3]
	\arrow[hook, from=1-3, to=1-4]
	\arrow[hook, from=1-4, to=1-5]
\end{tikzcd}
\end{equation}
	which is finite in each degree.
	Moreover, if $(p, b_\bullet)$ has degree exactly $j$,
	\begin{equation} \label{eq:assoc_graded_differential}
		\der \left( (p, b_\bullet) \pcirc_k \coarse_m^c \right) = (-)^{j+1-k}(p, b_\bullet) \pcirc_k \derbar \coarse_m^c \mod \filtr{k, j-1}{Q}_{j+m}
	\end{equation}
	in $\filtr{k, j}{Q}_\bullet$.
	Thus, the generators of the form $(p, b_\bullet) \pcirc_k \coarse_m^c$ for fixed $(p, b_\bullet)$ of degree $j$ span a subcomplex $\filtr{k, p, b_\bullet}{Q}_\bullet$ of $\filtr{k,j}{Q}_\bullet / \filtr{k, j-1}{Q}_\bullet$, such that
	\[
		\filtr{k, j}{Q}_\bullet / \filtr{k, j-1}{Q}_\bullet = \bigoplus_{(p, b_\bullet) \in \filtr{k-1}{\PNCP}(j)} \filtr{k, p, b_\bullet}{Q}_\bullet.
	\]
	Finally, let $\epsilon \eqdef (-)^{j+1-k}$. 
	By equation (\ref{eq:assoc_graded_differential}), the sequence of maps
	\[
		\coarse_m^c \mapsto \epsilon^m (p, b_\bullet) \pcirc_k \coarse_m^c
	\]
	establishes an isomorphism of chain complexes between $\filtr{1}{Q}_\bullet$ shifted by $j$ degrees and $\filtr{k, p, b_\bullet}{Q}_\bullet$.
	Since the former is acyclic, so is the latter.
	Thus the associated graded complex of $\filtr{k}{Q}_\bullet$ with respect to the filtration (\ref{eq:base_order_filtration}) is acyclic, hence---the filtration being finite in each degree---$\filtr{k}{Q}_\bullet$ is acyclic, and we conclude.
\end{proof}

\section{Towards the coherent self-dual \omegat-equivalence} \label{sec:equivalence}

\noindent
We recall some basic facts about strict $\omega$\nbd categories, partly in order to fix notations, which follow the conventions of \cite{hadzihasanovic2024combinatorics}.

\begin{dfn}[Strict $\omega$\nbd category]
	A \emph{strict $\omega$\nbd category} is a set $X$ together with, for all $n \in \NN$, \emph{boundary} operators $\bd{n}{-}, \bd{n}{+}\colon X \to X$, and, letting
	\[
		X \times_n X \eqdef \set{ (x, y) \in X \times X \mid \bd{n}{+}x = \bd{n}{-}y },
	\]
	\emph{$n$\nbd composition} operations $- \cp{n} -\colon X \times_n X \to X$, satisfying the axioms
	\begin{align*}
		& \text{there exists $n \in \NN$ such that $\bd{n}{-}x = \bd{n}{+}x = x$} 
			&& \text{(finite dimension)}, \\
		& \bd{n}{\a}\bd{k}{\b}x = 
			\begin{cases} 
				\bd{n}{\a}x & \text{if $n < k$}, \\
				\bd{k}{\b}x & \text{if $n \geq k$}
			\end{cases}
			&& \text{(globularity)}, \\
		& \bd{n}{\a}(x \cp{k} y) = 
			\begin{cases}
				\bd{n}{\a}x = \bd{n}{\a}y & \text{if $n < k$}, \\
				\bd{k}{-}x & \text{if $n = k$, $\a = -$}, \\
				\bd{k}{+}y & \text{if $n = k$, $\a = +$}, \\
				\bd{n}{\a} x \cp{k} \bd{n}{\a} y & \text{if $n > k$}
			\end{cases}
			&& \text{(boundary)}, \\
		& (x \cp{n} y) \cp{n} z = x \cp{n} (y \cp{n} z)
			&& \text{(associativity)}, \\
		& x \cp{n} \bd{n}{+}x = \bd{n}{-}x \cp{n} x = x
			&& \text{(unitality)}, \\
		& (x \cp{n+m} x') \cp{n} (y \cp{n+m} y') = (x \cp{n} y) \cp{n+m} (x' \cp{n} y')
			&& \text{(interchange)}
	\end{align*}
	for all $x, y, x', y', z \in X$, $\a, \b \in \set{-, +}$, $n, k \in \NN$, $m > 0$ such that the relevant compositions are defined.
\end{dfn}

\noindent The elements of a strict $\omega$\nbd category $X$ are called \emph{cells}.
Defining the \emph{dimension} of $x \in X$ to be
	\[
		\dim x \eqdef \min \set{n \in \NN \mid \bd{n}{-}x = \bd{n}{+}x = x}
	\]
induces a grading $X = \coprod_{n \in \NN} X_n$ on $X$.
If $x$ has dimension $n > 0$, we write $x\colon a \to b$ to specify that $a = \bd{n-1}{-}x$ and $b = \bd{n-1}{+}x$.
Notice that, in this presentation of strict $\omega$\nbd categories, there are no explicit unit operations; instead, the $n$\nbd boundaries of a cell can be cells of strictly lower dimension $k < n$.

For each $n \in \NN$, the \emph{$n$\nbd skeleton} of $X$ is the sub-$\omega$\nbd category $\skel{n}{X}$ obtained by restricting $X$ to its cells of dimension $\leq n$.
We say that $X$ is a \emph{strict $n$\nbd category} if it is equal to its $n$\nbd skeleton.

We let $\omegaCat$ denote the category of strict $\omega$\nbd categories and \emph{strict functors}, which are maps of the sets of cells which commute with boundaries and compositions, and $n\Cat$ its full subcategory on the strict $n$\nbd categories.
Restriction to the $n$\nbd skeleton determines a functor $\skel{n}\colon \omegaCat \to n\Cat$ which is right adjoint to the inclusion $\iota_n\colon n\Cat \incl \omegaCat$; there is also a left adjoint
\[
	\trunc{n}\colon \omegaCat \to n\Cat,
\]
called the \emph{$n$\nbd truncation}, or sometimes the ``intelligent'' $n$\nbd truncation \cite[\S 1.2]{ara2020joint}.
For each strict $\omega$\nbd category $X$, $\trunc{n}{X}$ is obtained as the quotient of $\skel{n}{X}$ by the equations $\bd{}{-}x = \bd{}{+}x$ for all $x \in X_{n+1}$.

There is a sequence $\globe{n}$, $n \in \NN$, of $\omega$\nbd categories called the \emph{globes}, with the property that strict functors $\globe{n} \to X$ correspond bijectively to cells of $\skel{n}{X}$.
For each $n \in \NN$, the $n$\nbd globe is a strict $n$\nbd category; we let $\bd{}{}\globe{n} \eqdef \skel{n-1}{\globe{n}}$.
Another description of the globes is that they are generated from the terminal category $1$ by iterated \emph{suspensions}.
The suspension $\sus{X}$ of a strict $\omega$\nbd category has exactly two 0\nbd cells $\bot^-, \bot^+$, its 1\nbd cells are of the form $\sus{x}\colon \bot^- \to \bot^+$ for $x$ a 0\nbd cell of $X$, and for $n > 0$, its $(n+1)$\nbd cells are of the form $\sus{x}\colon \sus{a} \to \sus{b}$ for $x\colon a \to b$ an $n$\nbd cell of $X$.

In \cite{lafont2010folk}, the authors construct model structures on $\omegaCat$, as well as on $n\Cat$ for all $n \in \NN$, such that the truncation functors are left Quillen; these are called the \emph{folk} model structures.
The cofibrant objects in the folk model structure on $\omegaCat$ are precisely \emph{polygraphs}: these are the strict $\omega$\nbd categories that are ``dimension-wise freely composition-generated'' by a set of cells.
More formally, $X$ is a polygraph if, for all $n \in \NN$, there exists a set of \emph{generating} $n$\nbd cells $\gener{X}_n \subseteq X_n$ such that
\[
\begin{tikzcd}
	{\coprod_{x \in \gener{X}_n} \bd{}{}\globe{n}} && {\coprod_{x \in \gener{X}_n} \globe{n}} \\
	\skel{n-1}{X} && \skel{n}{X}
	\arrow[from=1-1, to=2-1]
	\arrow["{(x)_{x \in {\gener{X}_n}}}", from=1-3, to=2-3]
	\arrow[hook, from=2-1, to=2-3]
	\arrow[hook, from=1-1, to=1-3]
	\arrow["\lrcorner"{anchor=center, pos=0.125, rotate=180}, draw=none, from=2-3, to=1-1]
\end{tikzcd}
\]
is a pushout diagram in $\omegaCat$; that is, $\skel{n}{X}$ is a ``cellular extension'' of $\skel{n-1}{X}$, with the inclusions $\bd{}{}\globe{n} \to \globe{n}$ as models of cells and cell boundaries.

\begin{dfn}[Linearisation of a strict $\omega$\nbd category]
	Let $X$ be a strict $\omega$\nbd category.
	The \emph{linearisation} of $X$ is the augmented chain complex $\linea{X}_\bullet$ with
	\[
		\linea{X}_n \eqdef \frac{
		\ZZ (\skel{n}{X}) } {
		\spanset{ t \cp{k} u - t - u \mid \text{$t, u \in \skel{n}{X}$, $k < n$} } }
	\]
	for each $n \in \NN$, with the differential maps defined on generators by
	\[
		x \mapsto \bd{n-1}{+}x - \bd{n-1}{-}x
	\]
	for all $x \in \skel{n}{X}$, and the augmentation $x \mapsto 1$ for all $x \in \skel{0}{X}$.
\end{dfn}

\noindent
In particular, when $X$ is a polygraph, for each $n \in \NN$ the group $\linea{X}_n$ is free on the set $\gener{X}_n$ of generating $n$\nbd cells.

We are ready to state our main conjecture.
To avoid redundant statements, we introduce the conventions
\[
	\filtr{\omega}{E_\bullet} \eqdef E_\bullet, \quad \filtr{\omega}{\PNCP}(n) \eqdef \PNCP(n), \quad \trunc{\omega} \eqdef \idd{\omegaCat}.
\]

\begin{conj} \label{conj:main_conjecture}
	There exists a polygraph $\omega E$, together with a filtration of cofibrations
\[\begin{tikzcd}
	{0 E} & {1 E} & \ldots & {k E} & \ldots
	\arrow[hook, from=1-1, to=1-2]
	\arrow[hook, from=1-2, to=1-3]
	\arrow[hook, from=1-3, to=1-4]
	\arrow[hook, from=1-4, to=1-5]
\end{tikzcd}
\]
whose colimit is $\omega E$, satisfying the following properties: for all $k \in \NN \cup \set{\omega}$,
\begin{enumerate}
		\item $k E$ has exactly two 0-cells $s, t$,
		\item for all $n \in \NN$, the set of generating $(n+1)$\nbd cells of $k E$ is in bijection with $\filtr{k}{\PNCP}(n)$,
		\item $\linea{(k E)}_\bullet$ is isomorphic to $\filtr{k}{E}_\bullet$,
		\item $\trunc{k}{(k E)}$ is contractible in the folk model structure on $k \Cat$.
	\end{enumerate}
	Moreover, there exists an involutive automorphism $\sigma\colon \omega E \to \omega E$ which swaps $s$ and $t$, and restricts to an automorphism of $k E$ for each $k > 0$, such that $\linea{(\sigma)} = \tau\colon E_\bullet \to E_\bullet$.
\end{conj}

\noindent
The properties would imply that, if $1$ is the terminal strict $\omega$\nbd category with a single 0\nbd cell, then the unique strict functor $1 \coprod 1 \to 1$ factors as the cofibration $1 \coprod 1 \incl \omega E$ classifying the pair of 0-cells $(s, t)$ followed by an acyclic fibration $\omega E \afib 1$; such an $\omega$\nbd category classifies $\omega$\nbd equivalences by \cite[Proposition 20.4.5]{ara2025polygraphs}, so $\omega E$ is a model of the coherent walking $\omega$\nbd equivalence, equivalent to the one produced in \cite{hadzihasanovic2025model}, but \emph{self-dual}.

\begin{comm}
	As first studied by Steiner \cite{steiner2004omega}, the linearisation functor can be seen as taking values in a category of augmented \emph{directed} chain complexes, where the abelian group in each degree is decorated with a distinguished submonoid of ``directed chains''.
	Steiner identified a class of augmented directed chain complexes, nowadays known as \emph{Steiner complexes}, on which the right adjoint $\nufun{}$ to linearisation is full and faithful and takes values in polygraphs.
	However, while $E_\bullet$ certainly admits a natural structure of augmented directed chain complex, it is as far away as possible from being a Steiner complex; in particular, $\nufun{E_\bullet}$ is not a polygraph.
\end{comm}

\noindent The rest of the article will consist of evidence for Conjecture \ref{conj:main_conjecture}.

\begin{evid}
	It follows from \cite[Theorem 6.1]{metayer2003resolutions} that, if $\omega E$ is contractible as a strict $\omega$\nbd category, then $\linea{(\omega E)}_\bullet$ must be acyclic.
	Thus $E_\bullet$ fits, at least, this requirement for the linearisation of a coherent walking $\omega$\nbd equivalence.
	While quasi-isomorphism of linearisations is much coarser than equivalence of strict $\omega$\nbd categories, this criterion could be used to disprove coherence of the ``naive'' walking $\omega$\nbd equivalence, see \cite[Proposition 1.3.2]{ozornova2024equivalence}. 
\end{evid}

\begin{evid}
	We give full explicit models for $0 E \incl 1 E$.
	A model for $0 E$ is the ``walking arrow'' 1\nbd category with two 0-cells $s, t$ and a single 1-cell $s \to t$, which we may label $\coarse_0^0$, the only noncrossing partition of level 0.
	Then $\linea{(0 E)}_\bullet = \filtr{0}{E}_\bullet$ is a simple check.
	Furthermore, $\tau_{\leq 0}(0 E)$ consists of a single 0\nbd cell and is a terminal object in $0\Cat \simeq \Set$.

	For $1 E$, we recall that there is a left adjoint functor $\orient\colon \sSet \to \omegaCat$ given by extending Street's cosimplicial object of \emph{orientals} \cite{street1987algebra} along colimits.
	This has the property that $\orient K$ is a polygraph whose generating $n$\nbd cells are in bijection with the nondegenerate $n$\nbd simplices of $K$, and the triangle of functors
\begin{equation} \label{eq:oriental_linearisation}
\begin{tikzcd}
	\sSet && \omegaCat \\
	& \chaug
	\arrow["\orient", from=1-1, to=1-3]
	\arrow["\NSimp"', from=1-1, to=2-2]
	\arrow["{\linea{}}", from=1-3, to=2-2]
\end{tikzcd}
\end{equation}
	commutes up to natural isomorphism.
	We let $1 E \eqdef \orient (\cosk_0 \set{s, t})$.
	Then Lemma \ref{lem:characterisation_of_E1} in combination with (\ref{eq:oriental_linearisation}) gives us an isomorphism between $\linea{(1 E)}_\bullet$ and $\filtr{1}{E}_\bullet$; moreover, the proof of Lemma \ref{lem:characterisation_of_E1} provides the bijection between generating cells of $1 E$ in dimension $n > 0$ and polarised noncrossing partitions of degree $n -1$ and level $\leq 1$.
	
	Let us look explicitly at the first few dimensions of $1 E$; we will identify generating cells in dimension $> 0$ with their corresponding polarised noncrossing partitions.
	To depict cells of dimension $\leq 2$, we will use 2\nbd categorical string diagrams as in \cite{hinze2023introducing}; to depict cells of dimension $> 2$, we will use pasting diagrams in hom-$\omega$\nbd categories between 2\nbd cells.
	Thus 3\nbd cells will be arrows between string diagrams and 4\nbd cells will be 2\nbd arrows between 3\nbd cells.
	In our string diagrams, 0\nbd composition proceeds left-to-right and 1\nbd composition proceeds bottom-to-top.

	We depict the 0\nbd generators $s$, $t$ as light, respectively dark regions:
\[
s = 
\begin{tikzpicture}[baseline={([yshift=-.5ex]current bounding box.center)},scale=.5]
	\path[fill, color=\lightregion] (-1,-1) rectangle (1,1);
\end{tikzpicture}\;, \quad
t =
\begin{tikzpicture}[baseline={([yshift=-.5ex]current bounding box.center)},scale=.5]
	\path[fill, color=\darkregion] (-1,-1) rectangle (1,1);
\end{tikzpicture}\;.
\]
	The 1\nbd generators $f \eqdef \coarse_0^0\colon s \to t$ and $g \eqdef \coarse_0^1\colon t \to s$ are boundaries between light and dark regions:
\[
f = 
\begin{tikzpicture}[baseline={([yshift=-.5ex]current bounding box.center)},scale=.5]
	\path[fill, color=\lightregion] (-1,-1) rectangle (1,1);
	\path[fill, color=\darkregion] (0,-1) rectangle (1,1);
\end{tikzpicture}\;, \quad
g =
\begin{tikzpicture}[baseline={([yshift=-.5ex]current bounding box.center)},scale=.5]
	\path[fill, color=\lightregion] (-1,-1) rectangle (1,1);
	\path[fill, color=\darkregion] (-1,-1) rectangle (0,1);
\end{tikzpicture}\;.
\]
The 2\nbd generators are $\pncp{ \ncstick{1} }\colon s \to f \cp{0} g$ and $\pncp{ \ncstick{1} \ncpol{1} }\colon t \to g \cp{0} f$:
\[
	\pncp{ \ncstick{1} } =
\begin{tikzpicture}[baseline={([yshift=-.5ex]current bounding box.center)},scale=.5]
	\path[fill, color=\lightregion] (-1,-1) rectangle (1,1);
	\path[fill, color=\darkregion] (-.75,1) to [out=-90,in=180] (0,.25) to [out=0,in=-90] (.75,1) -- cycle;
\end{tikzpicture}\;, \quad
\pncp{ \ncstick{1} \ncpol{1} } = 
\begin{tikzpicture}[baseline={([yshift=-.5ex]current bounding box.center)},scale=.5]
	\path[fill, color=\lightregion] (-1,-1) rectangle (1,1);
	\path[fill, color=\darkregion] (-1,-1) to (-1,1) to (-.75,1) to [out=-90,in=180] (0,.25) to [out=0,in=-90] (.75,1) to (1,1) to (1,-1) -- cycle;
\end{tikzpicture}\;.
\]
The 3\nbd generators are $\pncp{ \nccup{1}{2} } \colon \pncp{ \ncstick{1} } \cp{0} f \to f \cp{0} \pncp{ \ncstick{1} \ncpol{1} }$ and $\pncp{ \nccup{1}{2} \ncpol{1} }\colon \pncp { \ncstick{1} \ncpol{1} } \cp{0} g \to g \cp{0} \pncp{ \ncstick{1} }$:
\[\begin{tikzcd}
	{
\begin{tikzpicture}[baseline={([yshift=-.5ex]current bounding box.center)},scale=.5]
	\path[fill, color=\lightregion] (-1.5,-1) rectangle (1.5,1);
	\path[fill, color=\darkregion] (-1,1) to [out=-90,in=180] (-.5,.5) to [out=0,in=-90] (0,1) -- cycle;
	\path[fill, color=\darkregion] (1,1) to [out=-90,in=90] (0,-1) -- (1.5,-1) -- (1.5,1) -- cycle;
\end{tikzpicture}
	}
	 &&
	{ 
\begin{tikzpicture}[baseline={([yshift=-.5ex]current bounding box.center)},scale=.5]
	\path[fill, color=\lightregion] (-1.5,-1) rectangle (1.5,1);
	\path[fill, color=\darkregion] (0,-1) to [out=90,in=-90] (-1,1) to (0,1) to [out=-90,in=180] (.5,.5) to [out=0,in=-90] (1,1) -- (1.5,1) -- (1.5,-1) -- cycle;
\end{tikzpicture}\;,
	}
	\arrow["{\pncp{ \nccup{1}{2} }}", from=1-1, to=1-3]
\end{tikzcd}\]
\[\begin{tikzcd}
	{
\begin{tikzpicture}[baseline={([yshift=-.5ex]current bounding box.center)},scale=.5]
	\path[fill, color=\lightregion] (-1.5,-1) rectangle (1.5,1);
	\path[fill, color=\darkregion] (0,-1) to [out=90,in=-90] (1,1) to (0,1) to [out=-90,in=0] (-.5,.5) to [out=180,in=-90] (-1,1) -- (-1.5,1) -- (-1.5,-1) -- cycle;
\end{tikzpicture}
	}
	 &&
	{ 
\begin{tikzpicture}[baseline={([yshift=-.5ex]current bounding box.center)},scale=.5]
	\path[fill, color=\lightregion] (-1.5,-1) rectangle (1.5,1);
	\path[fill, color=\darkregion] (0,1) to [out=-90,in=180] (.5,.5) to [out=0,in=-90] (1,1) -- cycle;
	\path[fill, color=\darkregion] (-1,1) to [out=-90,in=90] (0,-1) -- (-1.5,-1) -- (-1.5,1) -- cycle;
\end{tikzpicture}\;.
	}
	\arrow["{\pncp{ \nccup{1}{2} \ncpol{1} }}", from=1-1, to=1-3]
\end{tikzcd}\]
The 4\nbd generators are 
\begin{align*}
	\pncp{ \nccup{1}{3} \ncstick{2} }\colon & \pncp{ \ncstick{1} } \cp{0} \pncp{ \ncstick{1} } \to \left(\pncp{ \ncstick{1} } \cp{1} \left(\pncp{ \nccup{1}{2} } \cp{0} g\right)\right) \cp{2} \left(\pncp{ \ncstick{1} } \cp{1} \left(f \cp{0} \pncp{ \nccup{1}{2} \ncpol{1} }\right)\right), \\
	\pncp{ \nccup{1}{3} \ncstick{2} \ncpol{1} }\colon & \pncp{ \ncstick{1} \ncpol{1} } \cp{0} \pncp{ \ncstick{1} \ncpol{1} } \to \left(\pncp{ \ncstick{1} \ncpol{1} } \cp{1} \left(\pncp{ \nccup{1}{2}\ncpol{1} } \cp{0} f\right)\right) \cp{2} \left(\pncp{ \ncstick{1} \ncpol{1} } \cp{1} \left(g \cp{0} \pncp{ \nccup{1}{2} }\right)\right),
\end{align*}
\[\begin{tikzcd}
	& 
	{
\begin{tikzpicture}[baseline={([yshift=-.5ex]current bounding box.center)},scale=.5]
	\path[fill, color=\lightregion] (-1.75,-1) rectangle (1.75,1);
	\path[fill, color=\darkregion] (-1.5,1) to [out=-90,in=180] (0,-.5) to [out=0,in=-90] (1.5,1) -- (.5, 1) to [out=-90, in=0] (0,.5) to [out=180, in=-90] (-.5,1) -- cycle;
\end{tikzpicture}
}& \\
	{
\begin{tikzpicture}[baseline={([yshift=-.5ex]current bounding box.center)},scale=.5]
	\path[fill, color=\lightregion] (-1.75,-1) rectangle (1.75,1);
	\path[fill, color=\darkregion] (-1.5,1) to [out=-90,in=180] (-1,.5) to [out=0,in=-90] (-.5,1) -- cycle;
	\path[fill, color=\darkregion] (.5,1) -- (.5,.5) to [out=-90,in=180] (1, 0) to [out=0, in=-90] (1.5,.5) -- (1.5,1) -- cycle;
\end{tikzpicture}
}
&& {
\begin{tikzpicture}[baseline={([yshift=-.5ex]current bounding box.center)},scale=.5]
	\path[fill, color=\lightregion] (-1.75,-1) rectangle (1.75,1);
	\path[fill, color=\darkregion] (.5,1) to [out=-90,in=180] (1,.5) to [out=0,in=-90] (1.5,1) -- cycle;
	\path[fill, color=\darkregion] (-1.5,1) -- (-1.5,.5) to [out=-90,in=180] (-1, 0) to [out=0, in=-90] (-.5,.5) -- (-.5,1) -- cycle;
\end{tikzpicture}\;,
}
	\arrow["{\pncp{ \ncstick{1} } \cp{1} (f \cp{0} \pncp{ \nccup{1}{2} \ncpol{1} })}", from=1-2, to=2-3]
	\arrow["{\pncp{ \ncstick{1} } \cp{1} (\pncp{ \nccup{1}{2} } \cp{0} g)}", from=2-1, to=1-2]
	\arrow[""{name=0, anchor=center, inner sep=0}, "{\pncp{ \ncstick{1} } \cp{0} \pncp{ \ncstick{1} }}"', from=2-1, to=2-3]
	\arrow["{\pncp{ \nccup{1}{3} \ncstick{2} }}"', between={0.2}{1}, Rightarrow, from=0, to=1-2]
\end{tikzcd}\]
\[\begin{tikzcd}
	& 
	{
\begin{tikzpicture}[baseline={([yshift=-.5ex]current bounding box.center)},scale=.5]
	\path[fill, color=\lightregion] (-1.75,-1) rectangle (1.75,1);
	\path[fill, color=\darkregion] (-.5, 1) to [out=-90, in=180] (0, .5) to [out=0, in=-90] (.5, 1) -- cycle;
	\path[fill, color=\darkregion] (-1.75, -1) -- (-1.75, 1) -- (-1.5, 1) to [out=-90, in=180] (0, -.5) to [out=0, in=-90] (1.5, 1) -- (1.75, 1) -- (1.75, -1) -- cycle;
\end{tikzpicture}
}& \\
	{
\begin{tikzpicture}[baseline={([yshift=-.5ex]current bounding box.center)},scale=.5]
	\path[fill, color=\lightregion] (-1.75,-1) rectangle (1.75,1);
	\path[fill, color=\darkregion] (-1.75, -1) -- (-1.75, 1) -- (-1.5, 1) to [out=-90, in=180] (-1, .5) to [out=0, in=-90] (-.5, 1) -- (.5, 1) -- (.5, .5) to [out=-90, in=180] (1, 0) to [out=9, in=-90] (1.5, .5) -- (1.5, 1) -- (1.75, 1) -- (1.75, -1) -- cycle;
\end{tikzpicture}
}
&& {
\begin{tikzpicture}[baseline={([yshift=-.5ex]current bounding box.center)},scale=.5]
	\path[fill, color=\lightregion] (-1.75,-1) rectangle (1.75,1);
	\path[fill, color=\darkregion] (-1.75, -1) -- (-1.75, 1) -- (-1.5, 1) -- (-1.5, .5) to [out=-90, in=180] (-1, 0) to [out=0, in=-90] (-.5, .5) -- (-.5, 1) -- (.5, 1) to [out=-90, in=180] (1, .5) to [out=0, in=-90] (1.5, 1) -- (1.75, 1) -- (1.75, -1) -- cycle;
\end{tikzpicture}\;.
}
	\arrow["{\pncp{ \ncstick{1} \ncpol{1} } \cp{1} (g \cp{0} \pncp{ \nccup{1}{2} })}", from=1-2, to=2-3]
	\arrow["{\pncp{ \ncstick{1} \ncpol{1} } \cp{1} (\pncp{ \nccup{1}{2} \ncpol{1} } \cp{0} f)}", from=2-1, to=1-2]
	\arrow[""{name=0, anchor=center, inner sep=0}, "{\pncp{ \ncstick{1} \ncpol{1} } \cp{0} \pncp{ \ncstick{1} \ncpol{1} }}"', from=2-1, to=2-3]
	\arrow["{\pncp{ \nccup{1}{3} \ncstick{2} \ncpol{1} }}"', between={0.2}{1}, Rightarrow, from=0, to=1-2]
\end{tikzcd}\]
One can observe that, indeed, the map $\sigma$ defined by $s \mapsto t$, $t \mapsto s$, and $\coarse_m^b \mapsto \coarse_m^{b \oplus 1}$ for all $m \in \NN$ and $b \in \set{0, 1}$ is an involutive automorphism of $1 E$; formally, this is induced by the functoriality of $\orient(\cosk_0 -)$ applied to the automorphism of $\set{s, t}$ exchanging $s$ and $t$.

The 1\nbd truncation $\trunc{1}{(1 E)}$ is the 1\nbd category generated by the 0-cells $s$, $t$ and the 1-cells $f\colon s \to t$, $g\colon t \to s$ subject to the equations $s = f \cp{0} g$ and $t = g \cp{0} f$ determined by the 2\nbd generators $\pncp{ \ncstick{1} }$ and $\pncp{ \ncstick{1} \ncpol{1} }$, respectively. 
This is precisely the \emph{walking isomorphism}, that is, the indiscrete groupoid on two objects, which as a category is equivalent to the terminal category, or equivalently, is contractible in the folk model structure on $1\Cat \simeq \Cat$.
We note that one can alternatively describe $\cosk_0 \set{s, t}$ as the nerve of the walking isomorphism. 
\end{evid}

\begin{evid} \label{evid:outergap}
Suppose that $(p, b_\bullet)$ is a polarised noncrossing partition of degree $k$, for which we have already given a model as a generating $(k+1)$\nbd cell of $\omega E$.
Then we can give an explicit model for the entire family of cells corresponding to the polarised noncrossing partitions
\begin{equation} \label{eq:suspension_family}
	(p, b_\bullet) \pcirc_{k+1} \coarse_m^c
\end{equation}
for $m \in \NN$ and $c \in \set{0, 1}$, as follows.
The cell $(p, b_\bullet)$ is classified by a strict functor $(p, b_\bullet)\colon \globe{k+1} \to \omega E$, which factors through the smallest sub-polygraph $X$ containing $(p, b_\bullet)$.
Now, $\globe{k+1}$ is isomorphic to the $k$\nbd fold suspension $\fun{S}^k \globe{1}$ of the 1\nbd globe, which is the walking arrow, that is, our $0 E$.
By functoriality of suspension applied to the inclusion $0 E \incl 1 E$, we have an inclusion $\fun{S}^k (0 E) \to \fun{S}^k (1 E)$.
We propose that the smallest subpolygraph $X'$ of $\omega E$ containing the entire family of cells (\ref{eq:suspension_family}) fits in a pushout diagram
\[\begin{tikzcd}
	{\fun{S}^k(0 E)} & {\fun{S}^k(1 E)} \\
	X & {X'}
	\arrow[hook, from=1-1, to=1-2]
	\arrow["{(p, b_\bullet)}", from=1-1, to=2-1]
	\arrow["{F_{(p, b_\bullet)}}", from=1-2, to=2-2]
	\arrow[hook, from=2-1, to=2-2]
	\arrow["\lrcorner"{anchor=center, pos=0.125, rotate=180}, draw=none, from=2-2, to=1-1]
\end{tikzcd}\]
of strict $\omega$\nbd categories; in particular, if $\coarse_m^c\colon O^{m+1} \to 1E$ classifies the generator $\coarse_m^c$ of $1 E$, the generator $(p, b_\bullet) \pcirc_{k+1} \coarse_m^c$ is classified by 
\[
	F_{(p, b_\bullet)} \circ \fun{S}^k(\coarse_m^c)\colon O^{m+k+1} \to X'.
\]
In particular, because $\coarse_0^0\colon s \to t$ and $\coarse_0^1\colon t \to s$ in $1 E$, if $(p, b_\bullet)\colon x \to y$ in $X$, then $(p, b_\bullet) \pcirc_{k+1} \coarse_0^1\colon y \to x$ in $X'$, exhibiting the weak inverse of $(p, b_\bullet)$.

This \emph{ansatz} allows us to systematically exhibit several other low-dimensional generators of $\omega E$ outside the image of $1 E$: for example, the 2\nbd generators
\[
	\pncp{ \ncstick{1} \ncpolfin{1} } =
\begin{tikzpicture}[baseline={([yshift=-.5ex]current bounding box.center)},scale=.5]
	\path[fill, color=\lightregion] (-1,-1) rectangle (1,1);
	\path[fill, color=\darkregion] (-.75,-1) to [out=90,in=180] (0,-.25) to [out=0,in=90] (.75,-1) -- cycle;
\end{tikzpicture}\;, \quad
\pncp{ \ncstick{1} \ncpol{1} \ncpolfin{1} } = 
\begin{tikzpicture}[baseline={([yshift=-.5ex]current bounding box.center)},scale=.5]
	\path[fill, color=\lightregion] (-1,-1) rectangle (1,1);
	\path[fill, color=\darkregion] (-1,1) to (-1,-1) to (-.75,-1) to [out=90,in=180] (0,-.25) to [out=0,in=90] (.75,-1) to (1,-1) to (1,1) -- cycle;
\end{tikzpicture}\;,
\]
as well as the 3\nbd generators
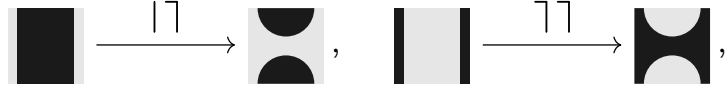
\[\begin{tikzcd}
	{
\begin{tikzpicture}[baseline={([yshift=-.5ex]current bounding box.center)},scale=.5]
	\path[fill, color=\lightregion] (-1,-1) rectangle (1,1);
\end{tikzpicture}
	}
	 &&
	{ 
\begin{tikzpicture}[baseline={([yshift=-.5ex]current bounding box.center)},scale=.5]
	\path[fill, color=\lightregion] (-1,-1) rectangle (1,1);
	\path[fill, color=\darkregion] (-.75,0) to [out=90,in=180] (0,.75) to [out=0,in=90] (.75,0) to [out=-90, in=0] (0, -.75) to [out=180, in=-90] (-.75, 0);
\end{tikzpicture}\;,
	}
	\arrow["{\pncp{ \ncstick{1} \ncstick{2} }}", from=1-1, to=1-3]
\end{tikzcd}
\quad
\begin{tikzcd}
	{
\begin{tikzpicture}[baseline={([yshift=-.5ex]current bounding box.center)},scale=.5]
	\path[fill, color=\darkregion] (-1,-1) rectangle (1,1);
\end{tikzpicture}
	}
	 &&
	{ 
\begin{tikzpicture}[baseline={([yshift=-.5ex]current bounding box.center)},scale=.5]
	\path[fill, color=\darkregion] (-1,-1) rectangle (1,1);
	\path[fill, color=\lightregion] (-.75,0) to [out=90,in=180] (0,.75) to [out=0,in=90] (.75,0) to [out=-90, in=0] (0, -.75) to [out=180, in=-90] (-.75, 0);
\end{tikzpicture}\;,
	}
	\arrow["{\pncp{ \ncstick{1} \ncpol{1} \ncstick{2} }}", from=1-1, to=1-3]
\end{tikzcd}\]
\[\begin{tikzcd}
	{
\begin{tikzpicture}[baseline={([yshift=-.5ex]current bounding box.center)},scale=.5]
	\path[fill, color=\lightregion] (-1,-1) rectangle (1,1);
	\path[fill, color=\darkregion] (-.75, -1) rectangle (.75, 1);
\end{tikzpicture}
	}
	 &&
	{ 
\begin{tikzpicture}[baseline={([yshift=-.5ex]current bounding box.center)},scale=.5]
	\path[fill, color=\lightregion] (-1,-1) rectangle (1,1);
	\path[fill, color=\darkregion] (-.75,-1) to [out=90,in=180] (0,-.25) to [out=0,in=90] (.75,-1) -- cycle;
	\path[fill, color=\darkregion] (-.75,1) to [out=-90,in=180] (0,.25) to [out=0,in=-90] (.75,1) -- cycle;
\end{tikzpicture}\;,
	}
	\arrow["{\pncp{ \ncstick{1} \ncstick{2} \ncpol{2} }}", from=1-1, to=1-3]
\end{tikzcd}
\quad
\begin{tikzcd}
	{
\begin{tikzpicture}[baseline={([yshift=-.5ex]current bounding box.center)},scale=.5]
	\path[fill, color=\lightregion] (-1,-1) rectangle (1,1);
	\path[fill, color=\darkregion] (-1, -1) rectangle (-.75, 1);
	\path[fill, color=\darkregion] (.75, -1) rectangle (1, 1);
\end{tikzpicture}
	}
	 &&
	{ 
\begin{tikzpicture}[baseline={([yshift=-.5ex]current bounding box.center)},scale=.5]
	\path[fill, color=\lightregion] (-1,-1) rectangle (1,1);
	\path[fill, color=\darkregion] (-1, -1) -- (-.75, -1) to [out=90, in=180] (0, -.25) to [out=0, in=90] (.75, -1) -- (1, -1) -- (1, 1) -- (.75, 1) to [out=-90, in=0] (0, .25) to [out=180, in=-90] (-.75, 1) -- (-1, 1) -- cycle;
\end{tikzpicture}\;,
	}
	\arrow["{\pncp{ \ncstick{1} \ncpol{1} \ncstick{2} \ncpol{2} }}", from=1-1, to=1-3]
\end{tikzcd}\]
and 4\nbd generators including
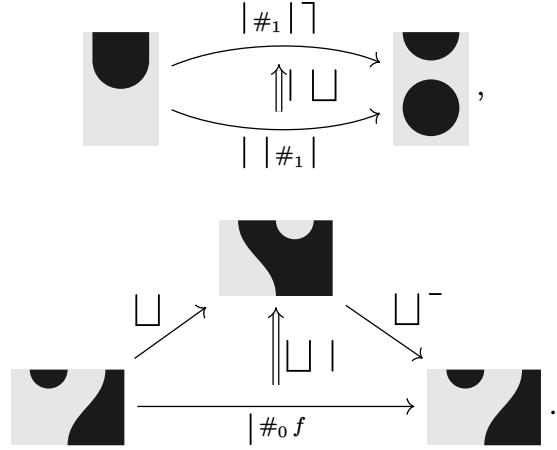
\[\begin{tikzcd}[column sep=large]
	{
\begin{tikzpicture}[baseline={([yshift=-.5ex]current bounding box.center)},scale=.5]
	\path[fill, color=\lightregion] (-1,-1.5) rectangle (1,1.5);
	\path[fill, color=\darkregion] (-.75, 1.5) -- (-.75,.75) to [out=-90,in=180] (0,0) to [out=0,in=-90] (.75,.75) -- (.75, 1.5) -- cycle;
\end{tikzpicture}
	} && {
\begin{tikzpicture}[baseline={([yshift=-.5ex]current bounding box.center)},scale=.5]
	\path[fill, color=\lightregion] (-1,-1.5) rectangle (1,1.5);
	\path[fill, color=\darkregion] (-.75,1.5) to [out=-90,in=180] (0,.75) to [out=0,in=-90] (.75,1.5) -- cycle;
	\path[fill, color=\darkregion] (-.75, -.5) to [out=90,in=180] (0,.25) to [out=0,in=90] (.75,-.5) to [out=-90, in=0] (0, -1.25) to [out=180, in=-90] (-.75, -.5);
\end{tikzpicture}\;,
	}
	\arrow[""{name=0, anchor=center, inner sep=0}, "{\pncp{ \ncstick{1} \ncstick{2} } \cp{1} \pncp{ \ncstick{1} }}"', curve={height=18pt}, from=1-1, to=1-3]
	\arrow[""{name=1, anchor=center, inner sep=0}, "{\pncp{ \ncstick{1} } \cp{1} \pncp{ \ncstick{1} \ncstick{2} \ncpol{2} }}", curve={height=-18pt}, from=1-1, to=1-3]
	\arrow["{\pncp{ \ncstick{1} \nccup{2}{3} }}"', between={0.2}{0.8}, Rightarrow, from=0, to=1]
\end{tikzcd}\]

\[\begin{tikzcd}
	& 
	{
\begin{tikzpicture}[baseline={([yshift=-.5ex]current bounding box.center)},scale=.5]
	\path[fill, color=\lightregion] (-1.5,-1) rectangle (1.5,1);
	\path[fill, color=\darkregion] (0,-1) to [out=90,in=-90] (-1,1) to (0,1) to [out=-90,in=180] (.5,.5) to [out=0,in=-90] (1,1) -- (1.5,1) -- (1.5,-1) -- cycle;
\end{tikzpicture}
}& \\
	{
\begin{tikzpicture}[baseline={([yshift=-.5ex]current bounding box.center)},scale=.5]
	\path[fill, color=\lightregion] (-1.5,-1) rectangle (1.5,1);
	\path[fill, color=\darkregion] (-1,1) to [out=-90,in=180] (-.5,.5) to [out=0,in=-90] (0,1) -- cycle;
	\path[fill, color=\darkregion] (1,1) to [out=-90,in=90] (0,-1) -- (1.5,-1) -- (1.5,1) -- cycle;
\end{tikzpicture}
}
&& {
\begin{tikzpicture}[baseline={([yshift=-.5ex]current bounding box.center)},scale=.5]
	\path[fill, color=\lightregion] (-1.5,-1) rectangle (1.5,1);
	\path[fill, color=\darkregion] (-1,1) to [out=-90,in=180] (-.5,.5) to [out=0,in=-90] (0,1) -- cycle;
	\path[fill, color=\darkregion] (1,1) to [out=-90,in=90] (0,-1) -- (1.5,-1) -- (1.5,1) -- cycle;
\end{tikzpicture}\;.
}
	\arrow["{\pncp{ \nccup{1}{2} \ncpolfin{2} }}", from=1-2, to=2-3]
	\arrow["{\pncp{ \nccup{1}{2} }}", from=2-1, to=1-2]
	\arrow[""{name=0, anchor=center, inner sep=0}, "{\pncp{ \ncstick{1} } \cp{0} f}"', from=2-1, to=2-3]
	\arrow["{\pncp{ \nccup{1}{2} \ncstick{3} }}"', between={0.2}{1}, Rightarrow, from=0, to=1-2]
\end{tikzcd}\]
\end{evid}

\noindent 
Beyond these cases, we do not have a systematic way of lifting generators of $E_\bullet$ to generators of $\omega E$.
It is clear, for example, that the duality induced on generators by $- \pcirc_k \coarse_0^1$ for each $k > 1$ has to do with direction-reversal of $k$\nbd cells; for example, the 3\nbd cells $\pncp{ \nccup{1}{2} \ncpol{2} }$ and $\pncp{ \nccup{1}{2} \ncpol{1} \ncpol{2} }$ should be vertical mirror images of $\pncp{ \nccup{1}{2} }$ and $\pncp{ \nccup{1}{2} \ncpol{1} }$, respectively:
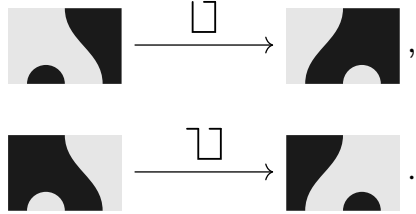
\[\begin{tikzcd}
	{
\begin{tikzpicture}[baseline={([yshift=-.5ex]current bounding box.center)},scale=.5, yscale=-1]
	\path[fill, color=\lightregion] (-1.5,-1) rectangle (1.5,1);
	\path[fill, color=\darkregion] (-1,1) to [out=-90,in=180] (-.5,.5) to [out=0,in=-90] (0,1) -- cycle;
	\path[fill, color=\darkregion] (1,1) to [out=-90,in=90] (0,-1) -- (1.5,-1) -- (1.5,1) -- cycle;
\end{tikzpicture}
	}
	 &&
	{ 
\begin{tikzpicture}[baseline={([yshift=-.5ex]current bounding box.center)},scale=.5, yscale=-1]
	\path[fill, color=\lightregion] (-1.5,-1) rectangle (1.5,1);
	\path[fill, color=\darkregion] (0,-1) to [out=90,in=-90] (-1,1) to (0,1) to [out=-90,in=180] (.5,.5) to [out=0,in=-90] (1,1) -- (1.5,1) -- (1.5,-1) -- cycle;
\end{tikzpicture}\;,
	}
	\arrow["{\pncp{ \nccup{1}{2} \ncpol{2} }}", from=1-1, to=1-3]
\end{tikzcd}\]
\[\begin{tikzcd}
	{
\begin{tikzpicture}[baseline={([yshift=-.5ex]current bounding box.center)},scale=.5, yscale=-1]
	\path[fill, color=\lightregion] (-1.5,-1) rectangle (1.5,1);
	\path[fill, color=\darkregion] (0,-1) to [out=90,in=-90] (1,1) to (0,1) to [out=-90,in=0] (-.5,.5) to [out=180,in=-90] (-1,1) -- (-1.5,1) -- (-1.5,-1) -- cycle;
\end{tikzpicture}
	}
	 &&
	{ 
\begin{tikzpicture}[baseline={([yshift=-.5ex]current bounding box.center)},scale=.5, yscale=-1]
	\path[fill, color=\lightregion] (-1.5,-1) rectangle (1.5,1);
	\path[fill, color=\darkregion] (0,1) to [out=-90,in=180] (.5,.5) to [out=0,in=-90] (1,1) -- cycle;
	\path[fill, color=\darkregion] (-1,1) to [out=-90,in=90] (0,-1) -- (-1.5,-1) -- (-1.5,1) -- cycle;
\end{tikzpicture}\;.
	}
	\arrow["{\pncp{ \nccup{1}{2} \ncpol{1} \ncpol{2} }}", from=1-1, to=1-3]
\end{tikzcd}\]
However, as the case of the pair $\pncp{ \ncstick{1} \ncstick{2} }$ and $\pncp{ \ncstick{1} \ncstick{2} \ncpol{2} }$ shows, this is not simply a case of the shape of $(p, b_\bullet)$ being dual to the shape of $(p, b_\bullet) \pcirc_k \coarse_0^1$.

\begin{evid}
	With the generators that we have exhibited so far, we are able to give a full description of $\trunc{2}{(2 E)}$ as the strict 2\nbd category generated by 0\nbd cells $s$, $t$, 1\nbd cells $f$, $g$, and 2\nbd cells $\pncp{ \ncstick{1} }$, $\pncp{ \ncstick{1} \ncpol{1} }$, $\pncp{ \ncstick{1} \ncpolfin{1} }$, $\pncp{ \ncstick{1} \ncpol{1} \ncpolfin{1} }$, subject to the equations
	\begin{equation} \label{eq:naive_equivalence} \begin{aligned}
		& s = \pncp{ \ncstick{1} } \cp{1} \pncp{ \ncstick{1} \ncpolfin{1} },
		&& t = \pncp{ \ncstick{1} \ncpol{1} } \cp{1} \pncp{ \ncstick{1} \ncpol{1} \ncpolfin{1} }, \\
		& f \cp{0} g = \pncp{ \ncstick{1} \ncpolfin{1} } \cp{1} \pncp{ \ncstick{1} },
		&& g \cp{0} f = \pncp{ \ncstick{1} \ncpol{1} \ncpolfin{1} } \cp{1} \pncp{ \ncstick{1} \ncpol{1} },
	\end{aligned} \end{equation}
	determined by the 3\nbd generators $\pncp{ \ncstick{1} \ncstick{2} }$, $\pncp{ \ncstick{1} \ncstick{2} \ncpol{1} }$, $\pncp{ \ncstick{1} \ncstick{2} \ncpol{2} }$, and $\pncp{ \ncstick{1} \ncstick{2} \ncpol{1} \ncpol{2} }$, respectively, as well as the equations
	\begin{equation} \label{eq:triangle_substitutes} \begin{aligned}
		& \pncp{ \ncstick{1} } \cp{0} f = f \cp{0} \pncp{ \ncstick{1} \ncpol{1} },
		&& \pncp{ \ncstick{1} \ncpol{1} } \cp{0} g = g \cp{0} \pncp{ \ncstick{1} }, \\
		& \pncp{ \ncstick{1} \ncpolfin{1} } \cp{0} f = f \cp{0} \pncp{ \ncstick{1} \ncpol{1} \ncpolfin{1} },
		&& \pncp{ \ncstick{1} \ncpol{1} \ncpolfin{1} } \cp{0} g = g \cp{0} \pncp{ \ncstick{1} \ncpolfin{1} }, 
	\end{aligned} \end{equation}
	determined by the 3\nbd generators $\pncp{ \nccup{1}{2} }$, $\pncp{ \nccup{1}{2} \ncpol{1} }$, $\pncp{ \nccup{1}{2} \ncpol{2} }$, $\pncp{ \nccup{1}{2} \ncpol{1} \ncpol{2} }$, respectively.
	The equations (\ref{eq:naive_equivalence}) determine a strict 2\nbd category isomorphic to the ``naive'' walking equivalence of 2\nbd categories, which is not contractible \cite[Proposition 1.3.2]{ozornova2024equivalence}.
	We claim that the equations (\ref{eq:triangle_substitutes}) make $\trunc{2}{(2 E)}$ isomorphic to the walking \emph{adjoint} equivalence, which is contractible \cite[Proposition 1.3.3]{ozornova2024equivalence}.
	
	With our notation for generators, the walking adjoint equivalence is obtained from the naive walking equivalence by imposing the triangle equations
	\[
		(\pncp{ \ncstick{1} } \cp{0} f) \cp{1} (f \cp{0} \pncp{ \ncstick{1} \ncpol{1} \ncpolfin{1} }) = f, \quad \quad
		(g \cp{0} \pncp{ \ncstick{1} }) \cp{1} (\pncp{ \ncstick{1} \ncpol{1} \ncpolfin{1} } \cp{0} g) = g,
	\]
	so it suffices to show that these are mutually derivable with (\ref{eq:triangle_substitutes}), given (\ref{eq:naive_equivalence}).
	Assuming (\ref{eq:naive_equivalence}) and (\ref{eq:triangle_substitutes}), we have
	\begin{align*}
		(\pncp{ \ncstick{1} } \cp{0} f) \cp{1} (f \cp{0} \pncp{ \ncstick{1} \ncpol{1} \ncpolfin{1} }) & = (f \cp{0} \pncp{ \ncstick{1} \ncpol{1} }) \cp{1} (f \cp{0} \pncp{ \ncstick{1} \ncpol{1} \ncpolfin{1} }) = \\	
													      & = f \cp{0} (\pncp{ \ncstick{1} \ncpol{1} } \cp{1} \pncp{ \ncstick{1} \ncpol{1} \ncpolfin{1} }) = f \cp{0} t = f,
	\end{align*}
	and dually for the other triangle equation.
	Assuming the triangle equations and (\ref{eq:naive_equivalence}), we have
	\begin{align*}
		\pncp{ \ncstick{1} } \cp{0} f  & = 
		(\pncp{ \ncstick{1} } \cp{0} f) \cp{1} (f \cp{0} g \cp{0} f) = \\
					       & = (\pncp{ \ncstick{1} } \cp{0} f) \cp{1} (f \cp{0} (\pncp{ \ncstick{1} \ncpol{1} \ncpolfin{1} } \cp{1} \pncp{ \ncstick{1} \ncpol{1} } )) = \\
					       & = (\pncp{ \ncstick{1} } \cp{0} f) \cp{1} (f \cp{0} \pncp{ \ncstick{1} \ncpol{1} \ncpolfin{1} }) \cp{1} (f \cp{0} \pncp{ \ncstick{1} \ncpol{1} }) = \\ 
					       & = f \cp{1}  (f \cp{0} \pncp{ \ncstick{1} \ncpol{1} }) =  f \cp{0} \pncp{ \ncstick{1} \ncpol{1} },
	\end{align*}
	as well as
	\begin{align*} 
		\pncp{ \ncstick{1} \ncpolfin{1} } \cp{0} f & =
		\pncp{ \ncstick{1} \ncpolfin{1} } \cp{0} ((\pncp{ \ncstick{1} } \cp{0} f) \cp{1} (f \cp{0} \pncp{ \ncstick{1} \ncpol{1} \ncpolfin{1} })) = \\
							   & = ((\pncp{ \ncstick{1} \ncpolfin{1}} \cp{1} \pncp{ \ncstick{1} }) \cp{0} f) \cp{1} (f \cp{0} \pncp{\ncstick{1} \ncpol{1} \ncpolfin{1} }) = \\
							   & = (f \cp{0} g \cp{0} f) \cp{1} (f \cp{0} \pncp{\ncstick{1} \ncpol{1} \ncpolfin{1}}) = f \cp{0} \pncp{\ncstick{1} \ncpol{1} \ncpolfin{1}}.
	\end{align*}
	The arguments for the other two equations are dual.
\end{evid}

\begin{rmk}
	Because of the known coherence result for the walking adjoint equivalence, hypothetical self-dual models of higher-dimensional coherent equivalences have mainly involved generators and equations coming from weakened notions of adjunction, such as the \emph{swallowtail} equations \cite[Definition 2.3]{gurski2012biequivalences}.
	In our conjectural model, the triangle equations, as well as these higher-dimensional cognates, are instead derived from ones that are sound for equivalences, but not for adjunctions.
	Thus this approach does not seem to offer clues towards algebraic presentations of higher-dimensional lax adjunctions.
\end{rmk}

Evidence \ref{evid:outergap} provided an interpretation of \emph{outer} gap-insertions of the $\coarse_m^c$---those of the form $(p, b_\bullet) \pcirc_{k+1} \coarse_m^c$ where $(p, b_\bullet)$ has degree $k$---as productions of families of higher coherence cells determined by the cell indexed by $(p, b_\bullet)$.
The main challenge towards the proof of Conjecture \ref{conj:main_conjecture} seems to be to provide a similar interpretation of \emph{inner} gap-insertions.
In low dimensions, we have been able to synthesise ``by hand'' coherence cells indexed by polarised noncrossing partitions with inner gap-insertions, extrapolating from the differential of $E_\bullet$.
The first case arises in dimension 4 with the generator $\pncp{ \nccup{1}{3} \ncstick[0.9]{2} }$ and its duals:
\[
	\der \pncp{ \nccup{1}{3} \ncstick[0.9]{2} } = \pncp{ \ncstick{1} \ncstick{2} \ncpol{1} } - 
	\pncp{ \ncstick{1} \ncstick{2} } - \pncp{ \nccup{1}{2} } - \pncp{ \nccup{1}{2} \ncpol{2} }
\]
in $E_\bullet$, which is compatible with the linearisation of the 4\nbd cell
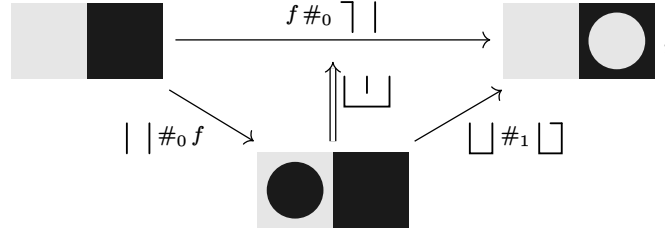
\[\begin{tikzcd}
	{
\begin{tikzpicture}[baseline={([yshift=-.5ex]current bounding box.center)},scale=.5]
	\path[fill, color=\lightregion] (-2,-1) rectangle (2,1);
	\path[fill, color=\darkregion] (0, -1) rectangle (2, 1);
\end{tikzpicture}
	} && {
\begin{tikzpicture}[baseline={([yshift=-.5ex]current bounding box.center)},scale=.5]
	\path[fill, color=\lightregion] (-2,-1) rectangle (2,1);
	\path[fill, color=\darkregion] (0, -1) rectangle (2, 1);
	\path[fill, color=\lightregion] (.25, 0) to [out=90, in=180] (1, .75) to [out=0, in=90] (1.75, 0) to [out=-90, in=0] (1, -.75) to [out=180, in=-90] (.25, 0);
\end{tikzpicture}\;,
	}	\\
	& {
\begin{tikzpicture}[baseline={([yshift=-.5ex]current bounding box.center)},scale=.5]
	\path[fill, color=\lightregion] (-2,-1) rectangle (2,1);
	\path[fill, color=\darkregion] (0, -1) rectangle (2, 1);
	\path[fill, color=\darkregion] (-1.75, 0) to [out=90, in=180] (-1, .75) to [out=0, in=90] (-.25, 0) to [out=-90, in=0] (-1, -.75) to [out=180, in=-90] (-1.75, 0);
\end{tikzpicture}
	}
	\arrow[""{name=0, anchor=center, inner sep=0}, "{f \cp{0} \pncp{ \ncstick{1} \ncstick{2} \ncpol{1} }}", from=1-1, to=1-3]
	\arrow["{\pncp{ \ncstick{1} \ncstick{2} } \cp{0} f}"', from=1-1, to=2-2]
	\arrow["{\pncp{ \nccup{1}{2} } \cp{1} \pncp{ \nccup{1}{2} \ncpol{2} }}"', from=2-2, to=1-3]
	\arrow["{\pncp{\nccup{1}{3} \ncstick[0.9]{2}}}"', between={0}{0.8}, Rightarrow, from=2-2, to=0]
\end{tikzcd}\]
encoding an evident coherence between the 3\nbd cells in its boundary.
However, we were not able, for now, to find an evident algebraic pattern in these extrapolations.
A possible path would be to try to generalise the construction of $1 E$ from the simplicial coskeleton $\cosk_0 \set{s, t}$, by finding some wider class of \emph{atoms} \cite{hadzihasanovic2024combinatorics}---perhaps generated by the oriented simplices under operations indexed by ``rightmost gap-insertions''---then consider the 0-coskeletal objects associated to a category of such cell shapes, as well as their realisation as strict $\omega$\nbd categories using the constructions of \cite[Section 5.2]{hadzihasanovic2024combinatorics}.
If this is possible, however, the operations are non-obvious, and so far our attempts have not been successful.

\bibliographystyle{alpha}
\small \bibliography{main}

@article{ebrahimi2020operads,
  title = {Operads of (noncrossing) partitions, interacting bialgebras, and moment-cumulant relations},
  author = {Ebrahimi-Fard, K. and Foissy, L. and Kock, J. and Patras, F.},
  journal = {Advances in Mathematics},
  volume = {369},
  pages = {107170},
  year = {2020},
}

@article{street1987algebra,
  title = {The algebra of oriented simplexes},
  author = {Street, R.},
  journal = {Journal of Pure and Applied Algebra},
  volume = {49},
  number = {3},
  pages = {283--335},
  year = {1987},
}

@article{steiner2004omega,
  title = {Omega-categories and chain complexes},
  author = {Steiner, R.},
  journal = {Homology, Homotopy and Applications},
  volume = {6},
  number = {1},
  pages = {175--200},
  year = {2004},
}

@article{hadzihasanovic2025model,
  title = {A model for the coherent walking $\omega$-equivalence},
  author = {Hadzihasanovic, A. and Loubaton, F. and Ozornova, V. and Rovelli, M.},
  journal={Proceedings of the American Mathematical Society},
  volume={153},
  number={07},
  pages={2813--2827},
  year={2025}
}

@book{hadzihasanovic2024combinatorics,
  title = {Combinatorics of higher-categorical diagrams},
  author = {Hadzihasanovic, A.},
  year = {2024},
  publisher = {Online preprint arXiv:2404.07273v2},
  note = {To appear in London Mathematical Society Lecture Note Series}
}

@book{goerss1999simplicial,
  title = {Simplicial Homotopy Theory},
  author = {Goerss, P.G. and Jardine, J.F.},
  series = {Progress in Mathematics},
  volume = {174},
  publisher = {Birkh\"auser},
  year = {1999},
}

@article{lafont2010folk,
  year = 2010,
  volume = {224},
  number = {3},
  pages = {1183--1231},
  author = {Lafont, Y. and M{\'{e}}tayer, F. and Worytkiewicz, K.},
  title = {A folk model structure on omega-cat},
  journal = {Advances in Mathematics}
}

@book{ara2020joint,
  title = {Joint et tranches pour les $\infty$-cat\'egories strictes},
  author = {Ara, D. and Maltsiniotis, G.},
  series = {M\'emoires de la Soci\'et\'e Math\'ematique de France},
  number = {165},
  publisher = {Soci\'et\'e Math\'ematique de France},
  year = {2020},
}

@book{ara2025polygraphs,
  series = {London Mathematical Society Lecture Note Series},
  title = {Polygraphs: From Rewriting to Higher Categories}, 
  publisher = {Cambridge University Press}, 
  author = {Ara, D. and Burroni, A. and Guiraud, Y. and Malbos, P. and M\'etayer, F. and Mimram, S.}, 
  year = {2025}, 
}

@article{metayer2003resolutions,
  title={Resolutions by polygraphs},
  author={M{\'e}tayer, F.},
  journal={Theory and Applications of Categories},
  volume={11},
  number={7},
  pages={148--184},
  year={2003}
}

@article{ozornova2024equivalence,
  title={What is an equivalence in a higher category?},
  author={Ozornova, V. and Rovelli, M.},
  journal={Bulletin of the London Mathematical Society},
  volume={56},
  number={1},
  pages={1--58},
  year={2024},
  publisher={Wiley Online Library}
}

@book{hinze2023introducing,
  title={Introducing string diagrams: the art of category theory},
  author={Hinze, R. and Marsden, D.},
  year={2023},
  publisher={Cambridge University Press}
}

@article{lack2004quillen,
  title={A {Q}uillen model structure for bicategories},
  author={Lack, S.},
  journal={K-theory},
  volume={33},
  number={3},
  pages={185--197},
  year={2004}
}

@article{gurski2012biequivalences,
  title={Biequivalences in tricategories},
  author={Gurski, N.},
  journal={Theory and Applications of Categories},
  volume={26},
  number={14},
  pages={349--384},
  year={2012}
}

@article{cheng2007omega,
  title = {An $\omega$-category with all duals is an $\omega$-groupoid},
  author = {Cheng, E.},
  journal = {Applied Categorical Structures},
  volume = {15},
  pages = {439--453},
  year = {2007},
  publisher = {Springer}
}

@article{fujii2024weakly,
  title={Weakly invertible cells in a weak $\omega$-category},
  author = {Fujii, S. and Hoshino, K. and Maehara, Y.},
  journal={Higher Structures},
  volume={8},
  number={2},
  pages={386--415},
  year={2024}
}

@unpublished{benjamin2024invertible,
  title = {Invertible cells in $\omega$-categories},
  author = {Benjamin, T. and Markakis, I.},
  note = {Online preprint arXiv:2406.12127},
  year = {2024}
}

@article{edelman1980chain,
  title={Chain enumeration and non-crossing partitions},
  author={Edelman, P.H.},
  journal={Discrete Mathematics},
  volume={31},
  number={2},
  pages={171--180},
  year={1980},
  publisher={Elsevier}
}

@article{chanavat2025equivalences,
  title={Equivalences in diagrammatic sets},
  author={Chanavat, C. and Hadzihasanovic, A.},
  journal={Journal of Pure and Applied Algebra},
  pages={108165},
  year={2025},
  publisher={Elsevier}
}

@article{kreweras1972partitions,
  title={Sur les partitions non crois{\'e}es d'un cycle},
  author={Kreweras, G.},
  journal={Discrete mathematics},
  volume={1},
  number={4},
  pages={333--350},
  year={1972},
  publisher={Elsevier}
}

\end{document}